\documentclass{article}
\usepackage[utf8]{inputenc}
\usepackage{amsmath}
\usepackage{amssymb}
\usepackage{color}
\usepackage{graphicx}
\usepackage{latexsym}
\usepackage{enumerate}
\usepackage{layout}
\usepackage{amsthm}
\usepackage[width=15.00cm]{geometry}
\usepackage{authblk}

\numberwithin{equation}{section}
\newtheorem{theorem}{Theorem}[section]
\newtheorem{prop}[theorem]{Proposition}
\newtheorem{lemma}[theorem]{Lemma}

\newtheorem{corollary}[theorem]{Corollary}
\newtheorem{proposition}[theorem]{Proposition}
\newtheorem{remark}[theorem]{Remark}

\newtheorem{condition}{Condition}[section]
\providecommand{\keywords}[1]{\textbf{Keywords and Phrases:} #1}
\providecommand{\classification}[1]{\textbf{AMS Classification:}#1}

\title{\huge Asymptotic Behaviour of Level Sets of \\ Needlet Random Fields}

\author[1]{ Radomyra Shevchenko}
\author[2]{Anna Paola Todino}
\affil[1]{\small Fakult\"at f\"ur Mathematik, Technische Universit\"at Dortmund}
\affil[2]{\small Fakult\"at f\"ur Mathematik, Ruhr-Universit\"at Bochum}

\date{}

\begin{document}

\maketitle

\begin{abstract}
\noindent 
We consider sequences of needlet random fields defined as weighted averaged forms of spherical Gaussian eigenfunctions. Our main result is a Central Limit Theorem in the high energy setting, for the boundary lengths of their excursion sets. 
This result is based on Stein-Malliavin techniques and Wiener chaos expansion for nonlinear functionals of random fields. To this end, a careful analysis of the variances of each chaotic component of the boundary length is carried out, showing that they are asymptotically constant, after normalisation, for all terms of the expansion and no leading component arises. 

\end{abstract}

\begin{itemize}
\item \keywords{} Gaussian spherical eigenfunctions, spherical needlets, boundary length, Wiener chaos, excursion sets, Central Limit Theorem.

\item \classification{} 60G60, 33C55, 62M15, 42C10, 60F05.
\end{itemize}

\maketitle

\section{Introduction and Main Result}
Spherical random fields constitute one of the central subjects in random geometry. It has been extensively studied in recent years not only due to its theoretical interest but also because of its importance in many applied domains, particularly in physics. Let us consider $\{f(x),\, x\in \mathbb{S}^2\}$, an isotropic Gaussian spherical random field with zero mean. It is known that $f(x)$ can be represented, in the $L^2$-sense, as
\[f(x)=\sum_{\ell=1}^{\infty}f_\ell(x)=\sum_{\ell=1}^{\infty}\sum_{m=-\ell}^\ell a_{\ell m}Y_{\ell m}(x),\]
where $\{a_{\ell m}\}$ are Gaussian random variables satisfying $\mathbb E [a_{\ell m}\bar{a}_{\ell'm'}]=C_\ell \delta_{\ell}^{\ell'}\delta_{m}^{m'}$ and  $\{Y_{\ell m}\}$ are the spherical harmonics  (see for instance \cite{MP}, \cite{Ma11}, \cite{Ma12}). The sequence $\{C_\ell\}_\ell$ is called angular power spectrum of $f(x)$ and is such that $\sum_{\ell} C_\ell \frac{2\ell+1}{4\pi}=\mathbb{E}f^2<\infty$ (see \cite{MP12}). The random fields $\{ f_\ell(x), x \in \mathbb{S}^2 \}$ are the eigenfunctions of the Laplace-Beltrami operator and hence they satisfy the Helmholtz equation $$\Delta_{\mathbb{S}^2}f_\ell+\ell(\ell+1)f_\ell=0,$$ $\ell\in\mathbb N$. They, too, are isotropic centered Gaussian with covariance function given by $$\mathbb{E}[f_\ell(x)f_\ell(y)]=C_\ell \frac{2\ell+1}{4\pi} P_\ell(\cos d(x,y)),$$ where $P_\ell$ is the Legendre polynomial of degree $\ell$ and $d(x,\,y)$ is the spherical distance between $x$ and $y$ on $\mathbb{S}^2$. 
The eigenfunctions $\{f_\ell\}_\ell$ have been widely investigated by many authors, in particular their excursion sets have been considered, for example, in \cite{W}, \cite{W1}, \cite{MRW}, \cite{MW11}, \cite{MR19}, \cite{CMW}.

However, in many practical situations, such as statistics for spherical observations, a Fourier analysis approach is not optimal if the data exhibits some blind spots. In this case procedures based on wavelet constructions on the sphere are preferred, in view of their  
 double localization properties in real and harmonic space. 
 Wavelet systems find applications, for instance, in astrophysics  and cosmology (see for example \cite{mcEwen}, \cite{carron}, \cite{Delabrouille}, \cite{MNRAS2008}, \cite{Oppizzi}, \cite{mcewen}). They are used to extract information from spherically observed signals in these fields, since the presence of a masked region in the domain of observation does lower the efficiency of the analysis. Spherical wavelet systems, so called needlets, have been introduced by \cite{NPW}, \cite{npw2} and investigated in the last years by many authors, see for instance \cite{bkmpAoSb}, \cite{LeGia}, \cite{Wang}, \cite{KerkyacharianNicklPicard}, \cite{BKMP09}, \cite{CM}, \cite{Durastanti}. In particular in \cite{BKMP09}, \cite{CM}, \cite{Durastanti} the properties of needlets applied to random fields are studied. This paper aims to give a contribution in this direction. 
 

To this purpose we consider an averaged form of the functions $f_\ell(x)$, defined as 
\begin{equation}\label{needlet}
\beta_j(x):=\sum_{\ell=2^{j-1}}^{2^{j+1}}b\left(\frac{\ell}{B^j} \right)f_\ell(x)
\end{equation}
for $j\in\mathbb N$. Here $B>1$ is a fixed parameter, called bandwidth, and $b(\cdotp)$ is a $C^{\infty}$ function with compact support on $[1\slash B,\, B]$ and satisfies the property $\sum_{j\in\mathbb N}b\left(\frac{\ell}{B^j} \right)^2 = 1$ for all $\ell>2$ (see \cite{MP}). \\
Let us consider the needlet kernel, as introduced by \cite{NPW}, defined for any $j=1,2,...$, as follows
\begin{equation*}
\Psi _{j}\left( x,y\right) =\sum_{\ell \geq
	0}b\left( \frac{\ell }{B^{j}}\right) \frac{2\ell+1}{4\pi}P_{\ell }\left( \left\langle
x,y\right\rangle \right) \text{ .}
\end{equation*}%
Then the random fields $\beta_j(x)$ can be seen as spherical needlets coefficients:
$$\beta_j(x)= \int_{\mathbb{S}^2} \Psi _{j}\left( x,y\right) f(y) dy.$$
The localization property of $\Psi _{j}\left( x,y\right) $, established in \cite{NPW}, is given by a bound which is nearly-exponential on $B^j$,  $j \in \mathbb{N}$. More precisely, we have that for all $(x,y)\in \mathbb{S}^{2}$ and
for all integers $M,$ there exists a constant $C_{M}$ such that%
\begin{equation}
\left\vert \Psi _{j}\left( x,y\right) \right\vert \leq \frac{C_{M}B^{2j}}{%
	\left\vert 1+B^{j}d (x,y)\right\vert ^{M}}\text{ .} 
\end{equation}
This property enables us to find a useful bound for the covariance function of the random field $\beta_j(x)$ (see Theorem 13.1 \cite{MP} and \cite{BKMP09}).\\

 In \cite{CM} the authors investigated the asymptotic behaviour of needlets polyspectra, defined as $$\int_{\mathbb{S}^2} H_q(\beta_j(x))\, dx,$$ where $H_q$ is the Hermite polynomial of order $q$. It turned out that the above integral has the same rate of convergence for each order $q$. Moreover, since the area of the excursion sets can be expanded in the $L^2$-sense through the Hermite polynomials (see also \cite{MW11} for details), applying some Stein-Malliavin results (the theory being detailed, for example, in \cite{NP}), the authors obtained a quantitive Central Limit Theorem (CLT) for the excursion area.\\

The purpose of our paper is to advance in the study of Lipschitz-Killing curvatures for needlet random fields. We investigate the boundary length of excursion sets of the normalised random field $\tilde\beta_j(x):=\beta_j(x)\slash \sqrt{\mathbb E [\beta_j(x)^2]}$, defined as 
\begin{equation}\label{length}
\mathcal L_j(z)=\{ x \in \mathbb{S}^2:  \tilde\beta_j(x)=z\},
\end{equation}
where $z \in \mathbb{R}$ is a fixed level. As for the area, the length $\mathcal{L}_j(z)$ can be expanded, in the $L^2-$ sense, in terms of its $q$-th order chaotic components to obtain the orthogonal expansion
$$\mathcal{L}_j(z)-\mathbb{E}[\mathcal{L}_j(z)]=\sum_{q=1}^{\infty} \operatorname{proj}[\mathcal{L}_j(z)|C_q],$$
$\operatorname{proj}[\mathcal{L}_\ell|C_q],$ denoting the projection on a subspace of $L^2$ called $q$th Wiener chaos (see Section \ref{Malliavin} and \cite{R}, \cite{MRW} for more details). 
The projection $\operatorname{proj}[\mathcal{L}_j(z)|C_q]$ involves integrals of Hermite polynomials of the form
\begin{equation}\label{int}
\int_{\mathbb{S}^2}H_{q-u}(\tilde{\beta}_j(x))H_k(\tilde\partial_1 \tilde{\beta}_j(x))H_{u-k}(\tilde\partial_2 \tilde{\beta}_j(x)) dx,
\end{equation}  for $u=0,\dots ,q, k=0,\dots ,u$; where 
$x=(x_1,x_2) \in \mathbb{S}^2$,
\begin{equation}\label{derivj}
\partial_i \beta_j(x) :=\sum_{\ell=2^{j-1}}^{2^{j+1}}b\left(\frac{\ell}{B^j} \right)\partial_i f_\ell(x)
\end{equation}
for $i=1,2$ and $\partial_i f_\ell(x):= \frac{\partial}{\partial x_i }f_\ell (x)$ (see \cite{MRW} for details) and $\tilde\partial_i \tilde{\beta}_j(x)$ being its renormalised version. The study of these integrals extends the findings on the polyspectra from \cite{CM} that emerge as a special case for $k=u=0$.
After a careful analysis of the variance of each chaotic component of $\mathcal L_j(z)$ we establish a Central Limit Theorem (CLT).
\begin{theorem}\label{Th1}
	Let $\mathcal L_j(z)$ be defined as in \eqref{length}, we have that, as $j \to \infty$,
	$$d_W\left(\frac{\mathcal L_j(z)-\mathbb E [\mathcal L_j(z)]}{\sqrt{\operatorname{Var}(\mathcal L_j(z))}}, N\right) \to 0$$
	where $N \sim N(0,\,1)$ and $d_W$ denotes the Wasserstein distance.
\end{theorem}
Note that since $\operatorname{proj}[\mathcal{L}_j(z)|C_q]$ is a sequence inside a fixed Wiener chaos (see Section \ref{Malliavin}), the celebrated Fourth Moment Theorem allows to establish a CLT for any single component. When dealing with the boundary length of the excursion sets for the eigenfunctions $f_\ell$ one can show that there is only a single leading component in the chaos expansion. In view of this the proof of the CLT for the whole series is immediate and was carried out in \cite{R}. This is not the case in our framework. Indeed, up to possible cancellation of the constant for discrete specific values of the level $z$, the variance of $\operatorname{proj}[\mathcal{L}_j(z)|C_q]$ is asymptotically constant for each $q$ after normalisation. 
 This analysis has been carried out differently for points far from the diagonal and for the ones close to it. For the latter we use an argument based on the two-point correlation function (see \cite{W}) showing that here the variance is finite. For points far from the diagonal we are able to prove that all the terms of the expansion have the same rate of convergence, using the Wiener chaos expansion and exploiting Hilb's asymptotic results. Therefore, no shortcut is possible in the study of the CLT and the entire series must be considered, similarly to the excursion area studied in \cite{CM}. Finally, making use of some Stein-Malliavin results (Theorem 5.1.3 \cite{NP}) and needlets localization properties, we prove Theorem \ref{Th1}.

\subsection{Plan of the paper}
In Section \ref{SectionNormalization} we introduce the main objects of the work, namely needlet random fields and their derivatives and derive useful results about their covariance functions. In Section \ref{Malliavin} we provide some background on Wiener chaoses and on the Stein-Malliavin method which is essential for the understanding of the paper. Section \ref{SecVariance} investigates the asymptotic variance of the boundary lengths for needlet coefficients, showing the asymptotic behaviour of every single Wiener chaos component of the chaotic expansion of the considered geometric functional. Finally the CLT is proved in Section \ref{SecCLT}. In the Appendix all the technical tools exploited in the main proofs are collected. More precisely, in \ref{Appkernel} we report and derive some properties of needlet kernels; in \ref{SecAuxiliary} other auxiliary results are recalled; \ref{Prooflem3prop3} and \ref{OnCLT} contain the proofs of some technical lemmas and propositions used in the study of the variance and the Central Limit Theorem, respectively.

\section{Needlet random fields and their derivatives}\label{SectionNormalization}
We start by considering the objects defined in (\ref{needlet}) and (\ref{derivj}); the bandwidth parameter from now on will be $B=2$. We first normalize all the random fields to have unit variance, hence, let us denote
\[\tilde{\beta}_j(x):=\frac{\beta_j(x)}{\sqrt{\mathbb E [\beta_j (x)^2]}},\]
where
\[\mathbb E [\beta_j (x)^2]=\sum_{\ell=2^{j-1}}^{2^{j+1}}b\left(\frac{\ell}{2^j}\right)^2C_\ell \frac{2\ell+1}{4\pi}=:B_{2^j}.\]
Moreover, since the random fields $f_\ell(x)$ are uncorrelated for different $\ell$ we get
\[\mathbb E [(\partial_1 \tilde{\beta}_j(x))^2]=\frac{1}{B_{2^j}}\sum_{\ell=2^{j-1}}^{2^{j+1}}b\left(\frac{\ell}{2^j}\right)^2\mathbb E [(\partial_1 f_\ell(x))^2]\]
and substituting $\mathbb E [(\partial_1 f_\ell(x))^2]=\frac{\ell(\ell+1)}{2}\mathbb E [f_\ell (x)^2]$ (see for example \cite{R}) we obtain
\begin{equation}\label{Aj}
\mathbb E [(\partial_1 \tilde{\beta}_j(x))^2]=\frac{1}{B_{2^j}}\sum_{\ell=2^{j-1}}^{2^{j+1}}b\left(\frac{\ell}{2^j}\right)^2\frac{\ell(\ell+1)}{2}C_\ell \frac{2\ell+1}{4\pi}=:A_{2^j}.
\end{equation}
Similarly to \cite{R}, it can be seen that also $\mathbb E [(\partial_2 \tilde{\beta}_j(x))^2]$ is equal to $A_{2^j}$.
Then we define
\[\tilde{\partial_1}\tilde{\beta}_j(x):=\frac{\partial_1 \tilde{\beta}_j(x)}{\sqrt{\mathbb E [(\partial_1 \tilde{\beta}_j(x))^2]}}\text{ and }\tilde{\partial_2}\tilde{\beta}_j(x):=\frac{\partial_2 \tilde{\beta}_j(x)}{\sqrt{\mathbb E [(\partial_2 \tilde{\beta}_j(x))^2]}}.\]

Let us now compute the covariance functions involved in this setting which we will exploit in the next sections. By isotropy permitting to fix the point $x=N$, where $N$ denotes the North Pole, and using spherical coordinates $(\theta,\varphi) \in [0,\pi]\times [0,2\pi)$, it was shown that for all $x,y\in\mathbb S^2$ 
\begin{eqnarray*}
\mathbb{E}[f_\ell(x)f_\ell(y)]&=& C_\ell\frac{2\ell+1}{4\pi} P_\ell(\cos \theta),\\
\mathbb{E}[f_\ell(x)\partial_1f_\ell(y)]&=& - C_\ell\frac{2\ell+1}{4\pi}P_\ell(\cos \theta)\sin\theta,\\
\mathbb{E}[\partial_1 f_\ell(x)\partial_1 f_\ell(y)]&=&  C_\ell\frac{2\ell +1}{4\pi}(P'_\ell(\cos \theta)\cos\theta - P''_\ell(\cos \theta)\sin^2\theta),\\
\mathbb{E}[\partial_2 f_\ell(x)\partial_2 f_\ell(y)]&=& C_\ell\frac{2\ell+1}{4\pi}P_\ell'(\cos \theta) ,
\end{eqnarray*}
where $\theta$ denotes the angle between $x$ and $y$ (see \cite{MRW} for details), which leads to
\begin{eqnarray}\label{covbeta}
\nonumber
	\tilde{\rho}_{1}(x,y)&:=&	\mathbb{E}[\tilde{\beta}_j(x)\tilde{\beta}_j(y)]= \frac{1}{B_{2^j}} \mathbb{E}[\beta_j(x)\beta_j(y)]= \frac{1}{B_{2^j}}\sum_{\ell=2^{j-1}}^{2^{j+1}}b\left(\frac{\ell}{2^j}\right)^2  C_\ell\frac{2\ell+1}{4\pi}P_\ell(\cos \theta),\\ \nonumber
	\tilde{\rho}_{2}(x,y)&:=&	\mathbb{E}[\tilde{\beta}_j(x)\tilde{\partial_1}\tilde{\beta}_j(y)]= \frac{1}{\sqrt{A_{2^j}}}\frac{1}{B_{2^j}} \mathbb{E}[\beta_j(x)\partial_1\beta_j(y)]\\ 
	&=&-\frac{1}{B_{2^j}\sqrt{A_{2^j}}}\sum_{\ell=2^{j-1}}^{2^{j+1}}b\left(\frac{\ell}{2^j}\right)^2 C_\ell\frac{2\ell +1}{4\pi}P'_\ell(\cos \theta)\sin\theta ,\\ \nonumber
	\tilde{\rho}_{3}(x,y)&:=&	\mathbb{E}[\tilde{\partial_1}\tilde{\beta}_j(x)\tilde{\partial_1}\tilde{\beta}_j(y)]= \frac{1}{A_{2^j}B_{2^j}} \mathbb{E}[\partial_1\beta_j(x)\partial_1\beta_j(y)]\\ \nonumber
	&=&\frac{1}{B_{2^j}A_{2^j}}\sum_{\ell=2^{j-1}}^{2^{j+1}}b\left(\frac{\ell}{2^j}\right)^2 C_\ell\frac{2\ell +1}{4\pi}( P'_\ell(\cos \theta)\cos\theta - P''_\ell(\cos \theta)\sin^2\theta), \\ \nonumber
	\tilde{\rho}_{4}(x,y)&:=&	\mathbb{E}[\tilde{\partial_2}\tilde{\beta}_j(x)\tilde{\partial_2}\tilde{\beta}_j(y)]= \frac{1}{A_{2^j}B_{2^j}} \mathbb{E}[\partial_2\beta_j(x)\partial_2\beta_j(y)] \\ \nonumber
	&=&\frac{1}{B_{2^j}A_{2^j}}\sum_{\ell=2^{j-1}}^{2^{j+1}}b\left(\frac{\ell}{2^j}\right)^2  C_\ell\frac{2\ell +1}{4\pi}P'_\ell(\cos \theta). 
\end{eqnarray}

In order to investigate the asymptotic behaviour of the boundary length for needlet random fields we assume some standard regularity conditions on the power spectrum $C_\ell$ (see \cite{MP}, page 257). 
\begin{condition}\label{cond1}
	There exists $M \in \mathbb{N}$, $a>4$ and a sequence of functions $\{g_j (\cdot)\}$ such that for $2^{j-1} < \ell <2^{j+1}$ 
	$$C_\ell = \ell^{-a} g_j\left(\frac{\ell}{2^j}\right) >0$$ where $c_0^{-1} \leq g_j\leq c_0$ for all $j \in \mathbb{N}$ and for some $c_1, \dots, c_M >0$ and $r=1, \dots, m$ we have $$\sup_j \sup_{2^{-1} \leq u \leq 2} \left|\frac{d}{du^r} g_j(u)\right| \leq c_r. $$
\end{condition}
Condition \ref{cond1} is fulfilled for example by models of the form $C_\ell=\ell^{-a}G(\ell)$ where $G(\ell)=P(\ell)/Q(\ell)$ and $P(\ell), Q(\ell) >0$ are two positive polynomials of the same order.


For such power spectrum models
it was shown in \cite{CM} that
\begin{equation}\label{asympBj}
\lim_{j\to\infty}(2^j)^{a-2} B_{2^j}=\frac{G}{2\pi}\int_{1\slash 2}^2 b^2(x)x^{1-a}dx,
\end{equation}
where $G$ is the limit of $G(\ell)$ for $\ell\to\infty$. In the following lemma we show that under the assumption of Condition \ref{cond1} a similar asymptotic result holds also for $A_{2^j}$.
\begin{lemma}\label{asympAa}
	Given $C_\ell =\ell^{-a}G(\ell)$ we have that
	\[\lim_{j\to\infty} (2^j)^{a -4}\sum_{\ell=2^{j-1}}^{2^{j+1}}b\left(\frac{\ell}{2^j}\right)^2\frac{\ell(\ell+1)}{2}C_\ell\frac{2\ell+1}{4\pi}=\frac{G}{4\pi}\int_{1\slash 2}^2 b^2(x)x^{3-a}dx.\]
\end{lemma}
\begin{proof}
	Note that
	\begin{equation*}
	    \begin{split}
		&\lim_{j\to\infty} (2^j)^{a -4}\sum_{\ell=2^{j-1}}^{2^{j+1}}b\left(\frac{\ell}{2^j}\right)^2\frac{\ell(\ell+1)}{2}C_\ell\frac{2\ell+1}{4\pi}\\
		&=\lim_{j\to\infty} (2^j)^{a-3}\sum_{\ell=2^{j-1}}^{2^{j+1}}\int_{\frac{\ell}{2^j}}^{\frac{\ell+1}{2^j}} b\left(\frac{\lfloor 2^j x\rfloor }{2^j}\right)^2\lfloor 2^j x\rfloor^{-a} G(\lfloor 2^j x\rfloor)\frac{2\lfloor 2^j x\rfloor +1}{4\pi}\frac{\lfloor 2^j x\rfloor(\lfloor 2^j x\rfloor+1)}{2}dx\\
		&=\lim_{j\to\infty}\int_{1\slash 2}^2  b\left(\frac{\lfloor 2^j x\rfloor }{2^j}\right)^2\left(\frac{\lfloor 2^j x\rfloor }{2^j}\right)^{-a}\frac{G(\lfloor 2^j x\rfloor)}{4\pi}\frac{2\lfloor 2^j x\rfloor +1}{2^{j+1}}\frac{\lfloor 2^j x\rfloor(\lfloor 2^j x\rfloor+1)}{2^{2j}}dx.
		\end{split}
	\end{equation*}
The result now follows by dominated convergence.
\end{proof}

Lemma \ref{asympAa} and (\ref{asympBj}) imply the following result.
\begin{corollary}\label{asympA} Let $A_j$ be defined as in (\ref{Aj}), we have that
\[\lim_{j\to\infty}(2^j)^{-2}A_{2^j}=\frac{1}{2} M_a\]
with $ M_a:= \frac{\int_{1/2}^{2}b^2(x)x^{3-a}\,dx}{\int_{1/2}^{2}b^2(x)x^{1-a}\,dx}$.
\end{corollary}
\begin{remark}\label{sum-asympt}
In the same way as in Lemma \ref{asympAa} it can be proved that for $p\in\mathbb R$:
	\[\lim_{j\to\infty} (2^j)^{a -2-p}\sum_{\ell=2^{j-1}}^{2^{j+1}}b\left(\frac{\ell}{2^j}\right)^2 C_\ell\frac{2\ell+1}{4\pi}(\ell^p+o(\ell^p))=\frac{G}{2\pi}\int_{1\slash 2}^2 b^2(x)x^{p+1-a}dx.\]
\end{remark}


\section{Malliavin calculus}\label{Malliavin}
In order to prove a Central Limit Theorem we will use results from Malliavin calculus established for functionals of Gaussian random processes. Below we give a brief overview over the main definitions and methods that will be applied (closely following \cite{NP}) and transfer the theory to our setting.

Let ${\mathcal{H}}$ be a real separable Hilbert space with its associated inner product ${\langle .,.\rangle}_{\mathcal{H}}$, and $(W (\varphi), \varphi\in{\mathcal{H}})$ an isonormal Gaussian process on a probability space $(\Omega , {\mathcal{F}}, P)$, which is a centered Gaussian family of random variables such that $\mathbb{E}\left( W(\varphi) W(\psi) \right) = {\langle\varphi,
\psi\rangle}_{{\mathcal{H}}}$ for every $\varphi,\psi\in{\mathcal{H}}$.

Note that for $\mathcal{H}=L^2(A,\mathcal{A}, \mu)$ over a Polish space $A$ with the associated $\sigma$-field $\mathcal{A}$ and a positive $\sigma$-finite and non-atomic measure $\mu$ one can define an isonormal Gaussian process with respect to the inner product $\langle g,h\rangle_{\mathcal H}=\int_A g(a)h(a)d\mu (a)$ as the Wiener-It\^o integral
\[W(h)=\int_A h(a)dW(a)\]
with respect to a Gaussian family $W=\{W(B): B\in\mathcal A , \mu(B)<\infty \}$ such that $\mathbb E [W(A)W(B)]=\mu (A\cap B)$ for $A$, $B$ of finite measure.

Denote by $I_{q}$ the $q$th multiple
stochastic integral with respect to $W$, that is, an
isometry between the Hilbert space ${\mathcal{H}}^{\odot q}$
(meaning symmetric tensor product) equipped with the norm
$\frac{1}{\sqrt{q!}}\Vert\cdot\Vert_{{\mathcal{H}}^{\otimes q}}$ and the Wiener chaos $C_q$ of order $q$, which is defined as the closed linear span of the random variables $H_{q}(W(\varphi))$,
$\varphi\in{\mathcal{H}}$ with $\Vert\varphi\Vert_{{\mathcal{H}}}=1$ and
$H_{q}$ the Hermite polynomial of degree $q\geq 1$ defined
by:
\begin{equation*}
H_{q}(x)=(-1)^{q} \exp \left( \frac{x^{2}}{2} \right) \frac{{\mathrm{d}}^{q}%
}{{\mathrm{d}x}^{q}}\left( \exp \left(
-\frac{x^{2}}{2}\right)\right),\; x\in \mathbb{R}.
\end{equation*}
For $f\in\mathcal H$ such that $\|f\|_{\mathcal H}=1$ we have for every $q\geq 1$
\[H_q(W(f))=I_q(f^{\otimes q})\]
and the isometry property can be written as follows: for $p,\;q\geq
1$,\;$f\in{{\mathcal{H}}^{\otimes p}}$ and
$g\in{{\mathcal{H}}^{\otimes q}}$
\begin{equation*}
\mathbb{E}\Big(I_{p}(f) I_{q}(g) \Big)=
\begin{cases}
q! \langle \tilde{f},\tilde{g}
\rangle _{{\mathcal{H}}^{\otimes q}} & \mbox{if}\;p=q,\\
 0 & \mbox{otherwise},
\end{cases}
\end{equation*}
where $\tilde{f} $ denotes the canonical symmetrization of $f$.

The space $L^2(\Omega,\mathcal F, P)$ can be decomposed in terms of Wiener chaoses: Every $F\in L^2(\Omega,\mathcal F, P)$ admits a unique expansion in $L^2$-sense
\[F=\mathbb E [F]+\sum_{q=1}^{\infty} F_q\]
with $F_q=\operatorname{proj}(F|C_q)$ belonging to the $q$th Wiener chaos.

Another definition that we will need is that of contraction on $\mathcal H = L^2(A,\mathcal{A}, \mu)$ For $f\in {\mathcal{H}}^{\otimes p}$, $g \in {\mathcal{H}}^{\otimes q}$ the $r$th contraction $f\otimes_{r}g$ is the element of
${\mathcal{H}}^{\otimes(p+q-2r)}$ which is defined by 
\begin{eqnarray*}
& (f\otimes_{r} g) ( s_{1}, \ldots, s_{p-r}, t_{1}, \ldots, t_{q-r})  \notag \\
& =\int_{A ^{r} } f( s_{1},
\ldots, s_{p-r}, u_{1}, \ldots,u_{r})g(t_{1}, \ldots, t_{q-r},u_{1},
\ldots,u_{r})\mathrm{d}\mu (u_{1})\ldots \mathrm{d}\mu (u_{r})
\end{eqnarray*}
for every $r=1,\ldots,p\wedge q$.

We have the following product formula for multiple integrals: if  $f\in{{\mathcal{H}}^{\odot p}}$ and
$g\in{{\mathcal{H}}^{\odot q}}$, then 
\begin{eqnarray}\label{product}
I_{p}(f) I_{q}(g)&=& \sum_{r=0}^{p \wedge q} r! \binom{p}{r}\binom{q}{r}I_{p+q-2r}\left(f\tilde{\otimes}_{r}g\right).
\end{eqnarray}

Let us denote by $D$ the Malliavin derivative operator that acts on
cylindrical random variables of the form $F=g(W(\varphi
_{1}),\ldots,W(\varphi_{n}))$, where $n\geq 1$,
$g:\mathbb{R}^n\rightarrow\mathbb{R}$ is a smooth function with at most polynomially growing derivatives and $\varphi_{i} \in {{\mathcal{H}}}$. This derivative is an element of $L^2(\Omega,{\mathcal{H}})$ and it is defined as
\begin{equation*}
DF=\sum_{i=1}^{n}\frac{\partial g}{\partial x_{i}}(W(\varphi _{1}),
\ldots , W(\varphi_{n}))\varphi_{i}.
\end{equation*}
This operator can be considered as inverse to the multiple integrals in the sense that for $q\geq 1$ and $f\in\mathcal H^{\odot q}$
\begin{equation}\label{derivM}
DI_q(f)=qI_{q-1}(f).
\end{equation}
The closure of the space of cylindrical random variables with respect to the norm
\[\sqrt{\mathbb E[F^2]+\mathbb E [\|DF\|^2_{\mathcal H}]}\]
is denoted by $\mathbb D^{1,2}$. 

A further operator that we need to consider is the Ornstein-Uhlenbeck operator $L$. For $F\in L^2(\Omega, \mathcal F, P)$ such that $\sum_{q\geq 1}q^2\mathbb E [F_q^2]<\infty$ it is defined as $LF=-\sum_{q\geq 1} qF_q$. Its pseudo-inverse $L^{-1}$ is an operator satisfying $L^{-1}F=-\sum_{q\geq 1}\frac{1}{q}F_q$.

The main tool for proving the Central Limit Theorem in this article is the following statement that generalises the celebrated Fourth Moment Theorem.
\begin{prop}[Theorem 5.1.3 in \cite{NP}]\label{Theorem5.1.3}
Let $F\in\mathbb D^{1,2}$ be centred and such that $\mathbb E[F^2]=\sigma^2>0$. Then for $N\sim N(0,\sigma^2)$ we have
\[d_W(F,N)\leq \frac{\sqrt{2}}{\sigma \sqrt{\pi}}\mathbb E\left[\left|\sigma^2-\langle DF, -DL^{-1}F\rangle_{\mathcal H}\right|\right],\]
where $d_W(F,N)$ is the Wasserstein distance between $F$ and $N$.
\end{prop}

Consider the Hilbert space $L^2(\mathbb{S}^2, d\sigma(x))$. As explained above, this space can be associated with an isonormal Gaussian process $W$ such that $\tilde\beta_j(x)$ can be expressed as an integral of the kernel $\tilde\Theta^{(0)}_j(\langle x,\cdot \rangle)$ which is defined as
\begin{eqnarray*}
\tilde\Theta^{(0)}_j(\langle x,y \rangle)=\frac{1}{\sqrt{B_{2^j}}}\sum_{\ell=2^{j-1}}^{2^{j+1}}b\left(\frac{\ell}{2^j}\right) \sqrt{C_\ell}\frac{2\ell+1}{4\pi}P_\ell(\langle x,\cdot \rangle).
\end{eqnarray*}
The random variables $\tilde\partial_1 \tilde{\beta}_j(x)$ and $\tilde\partial_2 \tilde{\beta}_j(x)$ are almost sure limits of elements of $\{W(h), \, h\in L^2(\mathbb{S}^2,d\sigma(x))\}$ and showing that they are also $L^2$ limits goes back to demonstrating that the integral and the differentiation in the deterministic formulas
\begin{eqnarray*}
\partial_i\int_{\mathbb{S}^2}\tilde\Theta^{(0)}_j(\langle x,y \rangle)d\sigma(y),\,i=1,\,2,
\end{eqnarray*}
are interchangeable. This can be easily seen by verifying this claim for Legendre polynomials. Consequently, we can write $\tilde\partial_1 \tilde{\beta}_j(x)$ and $\tilde\partial_2 \tilde{\beta}_j(x)$ as integrals with respect to $W$ of kernels $\tilde\Theta^{(1)}_j:=\frac{1}{\sqrt{A_{2^j}}}\partial_1 \tilde\Theta^{(0)}_j$ and $\tilde\Theta^{(2)}_j:=\frac{1}{\sqrt{A_{2^j}}}\partial_2 \tilde\Theta^{(0)}_j$ respectively.\\

Following a standard argument (explained for example in \cite{AT}), one can express the boundary length of level sets (for some level $z$) of $\tilde{\beta}_j$ as the almost sure limit of the integral
\[\mathcal L_j^{\varepsilon}(z):=\frac{1}{2\varepsilon}\int_{\mathbb{S}^2}1_{[z-\varepsilon,\,z+\varepsilon]}(\tilde{\beta}_j(x))\sqrt{(\partial_1 \tilde{\beta}_j(x))^2+(\partial_2 \tilde{\beta}_j(x))^2} dx,\]
which equals (with the normalisation from above)
\[\mathcal L_j^{\varepsilon}(z)=\sqrt{A_{2^j}}\frac{1}{2\varepsilon}\int_{\mathbb{S}^2}1_{[z-\varepsilon,\,z+\varepsilon]}(\tilde{\beta}_j(x))\sqrt{(\tilde\partial_1 \tilde{\beta}_j(x))^2+(\tilde\partial_2 \tilde{\beta}_j(x))^2} dx.\]
Following the same line of argument as was presented for $f_\ell$ in \cite{R}, we obtain that the convergence also takes place in $L^2$.

Relying on the fact that all random variables above are normed and subsequently using the same expansion and approximation as in \cite{R}, we arrive at the Hermite expansion of the level sets in terms of $\tilde{\beta}_j$ and its partial derivatives:
\begin{eqnarray*}
\operatorname{proj}(\mathcal L_j(z)|C_q)&=&\sqrt{A_{2^j}}\sum_{u=0}^q\sum_{k=0}^u\frac{\alpha_{k,\,u-k}\beta_{q-u}(z)}{k!(u-k)!(q-u)!}\\
&&\quad \times \int_{\mathbb{S}^2}H_{q-u}(\tilde{\beta}_j(x))H_k(\tilde\partial_1 \tilde{\beta}_j(x))H_{u-k}(\tilde\partial_2 \tilde{\beta}_j(x)) dx,
\end{eqnarray*}
where
\[\alpha_{2n,2m}=\sqrt{\frac{\pi}{2}}\frac{(2n)!(2m)!}{n!m!}\frac{1}{2^{n+m}}\sum_{j=0}^{n+m}(-1)^{j+n+m} \binom{n+m}{j}\frac{(2j+1)!}{(j!)^2}\frac{1}{4^j}\]
for an even pair of indices and zero otherwise and
\[\beta _\ell(z)=\phi (z)H_\ell(z)\]
with the standard Gaussian probability distribution function $\phi$ and the $\ell$-th Hermite polynomial $H_\ell$.

\section{Variance asymptotics of $\mathcal L_j(z)$}\label{SecVariance}
In this chapter we analyse the variance of $\mathcal L_j(z)$ in order to establish the proper normalisation for the Central Limit Theorem. The main object of our study will be in particular
\[\mathbb E [ \mathcal L_j(z)^2]=\lim_{\varepsilon_1,\varepsilon_2\to 0}\frac{1}{4\varepsilon_1\varepsilon_2}\int_{\mathbb{S}^2}\int_{\mathbb{S}^2}\mathbb E\left[ 1_{[z-\varepsilon_1, z+\varepsilon_2]}(\tilde\beta_j(x))\|\nabla\tilde{\beta}_j(x)\|1_{[z-\varepsilon_2, z+\varepsilon_2]}(\tilde\beta_j(y))\|\nabla\tilde{\beta}_j(y)\|\right] dxdy.\]
We use different methods depending on the relative position of the points $x$ and $y$ in this expression: we handle the case where they are close to each other using the so called $2$-point correlation function, and for the case where they are further apart we use the Wiener chaos decomposition and consider each summand separately. In fact, obtaining the order of $\mathbb E [ \mathcal L_j(z)^2]$ with respect to $j$ is rather straightforward: we calculate precisely the variance of $\operatorname{proj}(\mathcal L_j(z)|C_2)$ and show with a simple argument that the variances of all other chaos components are bounded by it. As mentioned in the introduction, in the manuscript \cite{R} in order to prove the Central Limit Theorem for the level sets of the random fields $f_\ell$, the author shows that the second chaos is dominating all other chaos components for $z\neq 0$ in the $L^2$-sense, making it the only factor contributing to the CLT. As this simplifies the proof, it makes sense to consider applying the same tactics in our case. However, this chapter also includes a result stating that (up to possible cancellation phenomena for specific values of $z$) every summand of every chaos component exhibits the same rate of convergence with respect to $j$ and must therefore be included in the final analysis. Moreover, this result helps analyse the variance more precisely.

We start with a variance calculation for the first and the second chaos component.

\begin{theorem}\label{2ndchaos}
	For $q=1$ we have  $\mathbb E[\operatorname{proj}(\mathcal L_j(z)|C_1)^2]=0$ and for $q=2$ we obtain
\[\lim_{j\to\infty}\mathbb E[\operatorname{proj}(\mathcal L_j(z)|C_2)^2]=\pi^3\phi(z)^2 M_a\frac{1}{\left(\int_{1\slash 2}^2 b^2(x)x^{1-a}dx\right)^2}\int_{1\slash 2}^2 b^4(x)x^{1-2a}(x^2 M_a^{-1}+z^2-1)^2dx,\]
	which is nonzero for all $z$.
\end{theorem}
\begin{proof}
	We have
	\begin{eqnarray*}
&&		\operatorname{proj}(\mathcal L_j(z)|C_1)=\sqrt{A_{2^j}}\alpha_{0,\,0}\beta_{1}(z) \int_{\mathbb{S}^2}H_{1}(\tilde{\beta}_j(x)) dx\\
&&\qquad=\sqrt{A_{2^j}}\alpha_{0,\,0}\beta_{1}(z) \int_{\mathbb{S}^2}\tilde{\beta}_j(x) dx=\sqrt{A_{2^j}}\alpha_{0,\,0}\beta_{1}(z)\sum_{\ell=2^{j-1}}^{2^j+1} b\left( \frac{\ell}{2^j} \right) \int_{\mathbb{S}^2}f_\ell(x)\, dx=0
	\end{eqnarray*}
by properties of spherical harmonics.\\
	
	For $q=2$ there are three summands for which the constant $\alpha_{k,\,u-k}$ is nonzero: the one with $u=k=0$, with $u=2,\,k=0$, and with $u=k=2$. This translates to
	\begin{eqnarray*}
		&&\operatorname{proj}(\mathcal L_j(z)|C_2)\\
		&&=\sqrt{A_{2^j}}\left(\frac{\alpha_{00}\beta_2(z)}{2}\int_{\mathbb S^2}H_2(\tilde{\beta}_j(x))dx+\frac{\alpha_{02}\beta_0(z)}{2}\int_{\mathbb S^2}H_2(\tilde{\partial}_2\tilde{\beta}_j(x))dx+\frac{\alpha_{20}\beta_0(z)}{2}\int_{\mathbb S^2}H_2(\tilde{\partial}_1\tilde{\beta}_j(x))dx\right)\\
		&&=\sqrt{A_{2^j}}\sqrt{\frac{\pi}{2}}\frac{1}{2}\phi (z)\left((z^2-1)\int_{\mathbb S^2}H_2(\tilde{\beta}_j(x))dx+\frac{1}{2}\int_{\mathbb S^2}H_2(\tilde{\partial}_2\tilde{\beta}_j(x))dx+\frac{1}{2}\int_{\mathbb S^2}H_2(\tilde{\partial}_1\tilde{\beta}_j(x))dx\right),
	\end{eqnarray*}
	and we can write
	\begin{eqnarray*}
		\mathbb E[\operatorname{proj}(\mathcal L_j(z)|C_2)^2]=A_{2^j}\frac{\pi}{8}\phi(z)^2(S_1+S_2+S_3+S_4),
	\end{eqnarray*}
	where
	\begin{eqnarray*}
		S_1&=&\int_{\mathbb S^2}\int_{\mathbb S^2}\mathbb E \left[(z^2-1)^2H_2(\tilde{\beta}_j(x))H_2(\tilde{\beta}_j(y))\right]dxdy,\\	S_2&=&2\int_{\mathbb S^2}\int_{\mathbb S^2}\mathbb E \left[(z^2-1)\frac{1}{2}H_2(\tilde{\beta}_j(x))H_2(\tilde{\partial}_1\tilde{\beta}_j(y))\right]dxdy,\\
		S_3&=&\int_{\mathbb S^2}\int_{\mathbb S^2}\mathbb E \left[\frac{1}{4}H_2(\tilde{\partial}_2\tilde{\beta}_j(x))H_2(\tilde{\partial}_2\tilde{\beta}_j(y))\right]dxdy,\\
		S_4&=&\int_{\mathbb S^2}\int_{\mathbb S^2}\mathbb E \left[\frac{1}{4}H_2(\tilde{\partial}_1\tilde{\beta}_j(x))H_2(\tilde{\partial}_1\tilde{\beta}_j(y))\right]dxdy.
	\end{eqnarray*}
	Since $A_{2^j}$ is of order $(2^j)^2$, for obtaining the convergence result it suffices to show that all four summands are of order $2^{-2j}$. Let us begin with $S_1$. We write by the diagram formula (see Proposition \ref{diag} in the appendix) and \eqref{covbeta}
	\begin{eqnarray*}
		S_1&=&(z^2-1)^2\int_{\mathbb S^2}\int_{\mathbb S^2}2\mathbb E[\tilde{\beta}_j(x)\tilde{\beta}_j(y)]^2dxdy\\
		&=&(z^2-1)^2 8\pi^2\frac{2}{B_{2^j}^2}\int_0^\pi \left(\sum_{\ell=2^{j-1}}^{2^{j+1}}b\left(\frac{\ell}{2^j}\right)^2 C_\ell \frac{2\ell+1}{4\pi}P_\ell(\cos \theta)\right)^2\sin \theta d\theta\\
		&=&(z^2-1)^2 8\pi^2\frac{2}{B_{2^j}^2}\sum_{\ell_1,\ell_2=2^{j-1}}^{2^{j+1}}b\left(\frac{\ell_1}{2^j}\right)^2 C_{\ell_1} \frac{2\ell_1+1}{4\pi}b\left(\frac{\ell_2}{2^j}\right)^2 C_{\ell_2} \frac{2\ell_2+1}{4\pi}\int_0^\pi P_{\ell_1}(\cos \theta) P_{\ell_2}(\cos \theta)\sin\theta d\theta\\
		&=&(z^2-1)^2 8\pi^2\frac{2}{B_{2^j}^2}\sum_{\ell=2^{j-1}}^{2^{j+1}}b\left(\frac{\ell}{2^j}\right)^4 C_{\ell}^2 \left(\frac{2\ell+1}{4\pi}\right)^2\frac{2}{2\ell+1}
	\end{eqnarray*}
	by orthogonality. As one can see by Remark \ref{sum-asympt}, this term is of order $2^{-2j}$. Similarly, for $S_2$ we obtain
	\begin{eqnarray*}
		S_2&=&(z^2-1)\int_{\mathbb S^2}\int_{\mathbb S^2}2\mathbb E[\tilde{\beta}_j(x)\tilde{\partial}_1\tilde{\beta}_j(y)]^2dxdy\\
		&=&(z^2-1)8\pi^2\frac{2}{B_{2^j}^2A_{2^j}}\sum_{\ell_1,\ell_2=2^{j-1}}^{2^{j+1}}b\left(\frac{\ell_1}{2^j}\right)^2 C_{\ell_1} \frac{2\ell_1+1}{4\pi}b\left(\frac{\ell_2}{2^j}\right)^2 C_{\ell_2} \frac{2\ell_2+1}{4\pi}\int_0^\pi P'_{\ell_1}(\cos \theta) P'_{\ell_2}(\cos \theta)\sin^3\theta d\theta\\
		&=&(z^2-1)16\pi^2\frac{1}{B_{2^j}^2A_{2^j}}\sum_{\ell_1,\ell_2=2^{j-1}}^{2^{j+1}}b\left(\frac{\ell_1}{2^j}\right)^2 C_{\ell_1} \frac{2\ell_1+1}{4\pi}b\left(\frac{\ell_2}{2^j}\right)^2 C_{\ell_2} \frac{2\ell_2+1}{4\pi}\int_{-1}^1 P'_{\ell_1}(x) P'_{\ell_2}(x)(1-x^2)dx\\
		&=&(z^2-1)16\pi^2\frac{1}{B_{2^j}^2A_{2^j}}\sum_{\ell_1,\ell_2=2^{j-1}}^{2^{j+1}}b\left(\frac{\ell_1}{2^j}\right)^2 C_{\ell_1} \frac{2\ell_1+1}{4\pi}b\left(\frac{\ell_2}{2^j}\right)^2 C_{\ell_2} \frac{2\ell_2+1}{4\pi}\int_{-1}^1 P^1_{\ell_1}(x) P^1_{\ell_2}(x)dx\\
	\end{eqnarray*}
	with $P^m_n$ denoting the associated Legendre functions. Note that for two such functions of the same degree $m$ an orthogonality relation holds, namely
	\begin{eqnarray*}
		\int_{-1}^1 P^m_\ell(x)P^m_k(x)dx=\frac{2(l+m)!}{(2\ell+1)(\ell-m)!}\delta_{k\ell}.
	\end{eqnarray*}
	Therefore, we have
	\begin{eqnarray*}
		S_2&=&(z^2-1)16\pi^2\frac{1}{B_{2^j}^2A_{2^j}}\sum_{\ell=2^{j-1}}^{2^{j+1}}b\left(\frac{\ell}{2^j}\right)^4 C_{\ell}^2 \left(\frac{2\ell+1}{4\pi}\right)^2\frac{2\ell(\ell+1)}{2\ell+1}.
	\end{eqnarray*}
	This term is also of order $2^{-2j}$. For $S_3$ we calculate
	\begin{eqnarray*}
		S_3&=&\frac{2}{4}\int_{\mathbb S^2}\int_{\mathbb S^2}\mathbb E[\tilde\partial_2\tilde{\beta}_j(x)\tilde{\partial}_2\tilde{\beta}_j(y)]^2dxdy\\
		&=&4\pi^2\frac{1}{B_{2^j}^2A_{2^j}^2}\sum_{\ell_1,\ell_2=2^{j-1}}^{2^{j+1}}b\left(\frac{\ell_1}{2^j}\right)^2 C_{\ell_1} \frac{2\ell_1+1}{4\pi}b\left(\frac{\ell_2}{2^j}\right)^2 C_{\ell_2} \frac{2\ell_2+1}{4\pi}\int_0^\pi P'_{\ell_1}(\cos \theta) P'_{\ell_2}(\cos \theta)\sin\theta d\theta\\
		&=&4\pi^2\frac{1}{B_{2^j}^2A_{2^j}^2}\sum_{\ell_1,\ell_2=2^{j-1}}^{2^{j+1}}b\left(\frac{\ell_1}{2^j}\right)^2 C_{\ell_1} \frac{2\ell_1+1}{4\pi}b\left(\frac{\ell_2}{2^j}\right)^2 C_{\ell_2} \frac{2\ell_2+1}{4\pi}\int_{-1}^1 P'_{\ell_1}(x) P'_{\ell_2}(x) dx\\
		&\leq& 4\pi^2\frac{1}{B_{2^j}^2A_{2^j}^2}\sum_{\ell_1,\ell_2=2^{j-1}}^{2^{j+1}}b\left(\frac{\ell_1}{2^j}\right)^2 C_{\ell_1} \frac{2\ell_1+1}{4\pi}b\left(\frac{\ell_2}{2^j}\right)^2 C_{\ell_2} \frac{2\ell_2+1}{4\pi} 2 (\ell_1+1)\\
		&=&4\pi^2\frac{1}{B_{2^j}^2A_{2^j}^2}\sum_{\ell_1=2^{j-1}}^{2^{j+1}}b\left(\frac{\ell_1}{2^j}\right)^2 C_{\ell_1} \frac{2\ell_1+1}{4\pi}2 (\ell_1+1) \sum_{\ell_2=2^{j-1}}^{2^{j+1}}b\left(\frac{\ell_2}{2^j}\right)^2 C_{\ell_2} \frac{2\ell_2+1}{4\pi}
	\end{eqnarray*}
	by Proposition \ref{deriv}. This term is of order $2^{-3j}$. For $S_4$ note first that by the characterising differential equation for Legendre polynomials we have
	\begin{eqnarray*}
		P'_\ell(x)x-P''_\ell(x)(1-x^2)=\ell(\ell+1)P_\ell(x)-xP'_\ell(x).
	\end{eqnarray*}
	Moreover, due to the identity $xP'_\ell(x)=P_{\ell+1}'(x)-(\ell+1)P_\ell(x)$ we obtain the identity
	\begin{eqnarray*}
		P'_\ell(x)x-P''_\ell(x)(1-x^2)=(\ell+1)^2P_\ell(x)-P'_{\ell+1}(x).
	\end{eqnarray*}
	Therefore, the integral
	\begin{eqnarray*}
		\int_{-1}^1 (P'_{\ell_1}(x)x-P''_{\ell_1}(x)(1-x^2))(P'_{\ell_2}(x)x-P''_{\ell_2}(x)(1-x^2))dx
	\end{eqnarray*}
	can be decomposed as
	\begin{eqnarray*}
		&&(\ell_1+1)^2(\ell_2+1)^2\int_{-1}^1 P_{\ell_1}(x)P_{\ell_2}(x)dx-(\ell_1+1)^2\int_{-1}^1P_{\ell_1}(x)P'_{\ell_2+1}(x)dx\\
		&&\quad-(\ell_2+1)^2\int_{-1}^1P_{\ell_2}(x)P'_{\ell_1+1}(x)dx+\int_{-1}^1P'_{\ell_1+1}(x)P'_{\ell_2+1}(x)dx.
	\end{eqnarray*}
	The first summand equals $(\ell_1+1)^2(\ell_2+1)^2\frac{2}{2\ell_1+1}\delta_{\ell_1\ell_2}$ by orthogonality, the second and third can be bounded (in absolute value) by $(\ell_1+1)^2$ and $(\ell_2+1)^2$ respectively by Proposition \ref{deriv}, and the last summand's absolute value is bounded by $2(\ell_1+2)$ due to the same proposition. The remaining calculations for $S_4$ are now straightforward:
	\begin{eqnarray*}
		S_4&=&\frac{1}{4}\int_{\mathbb S^2}\int_{\mathbb S^2}2\mathbb E[\tilde\partial_1\tilde{\beta}_j(x)\tilde{\partial}_1\tilde{\beta}_j(y)]^2dxdy\\
		&=&4\pi^2\frac{1}{B_{2^j}^2A_{2^j}^2}\sum_{\ell_1,\ell_2=2^{j-1}}^{2^{j+1}}b\left(\frac{\ell_1}{2^j}\right)^2 C_{\ell_1} \frac{2\ell_1+1}{4\pi}b\left(\frac{\ell_2}{2^j}\right)^2 C_{\ell_2} \frac{2\ell_2+1}{4\pi}\\
		&&\qquad \times 	\int_{-1}^1 (P'_{\ell_1}(x)x-P''_{\ell_1}(x)(1-x^2))(P'_{\ell_2}(x)x-P''_{\ell_2}(x)(1-x^2))dx
	\end{eqnarray*}
	by the usual change of variables. The first of the four summands arising from the decomposition of this integral indicated above equals
	\begin{eqnarray*}
		4\pi^2\frac{1}{B_{2^j}^2A_{2^j}^2}\sum_{\ell=2^{j-1}}^{2^{j+1}}b\left(\frac{\ell}{2^j}\right)^4 C_{\ell}^2 \left(\frac{2\ell+1}{4\pi}\right)^2(\ell+1)^4\frac{2}{2\ell+1},
	\end{eqnarray*}
	and it is of order $2^{-2j}$. The other three summands have bounds that converge even faster.
	We can calculate the exact limiting variance of the second chaos component by calculating the limits of $S_1$, $S_2$ and the slowest summand of $S_4$. For a level $z\in\mathbb R$ it equals
$$\pi^3\phi(z)^2 M_a\frac{1}{\left(\int_{1\slash 2}^2 b^2(x)x^{1-a}dx\right)^2}\int_{1\slash 2}^2 b^4(x)x^{1-2a}(x^2 M_a^{-1}+z^2-1)^2dx,$$
which is clearly nonzero for all $z$.
\end{proof}

To analyse the variances of higher chaoses we will need a more general application of the diagram formula for dealing with the expression
$$\mathbb E [H_{q-u_1}(\tilde{\beta}_j(x))H_{k_1}(\tilde\partial_1 \tilde{\beta}_j(x))H_{u_1-k_1}(\tilde\partial_2 \tilde{\beta}_j(x))H_{q-u_2}(\tilde{\beta}_j(y))H_{k_2}(\tilde\partial_1 \tilde{\beta}_j(y))H_{u_2-k_2}(\tilde\partial_2 \tilde{\beta}_j(y))].$$
In this context we have the following lemma.

\begin{lemma}\label{constdiag}
We have for $x=(0,0)$ and $y=(0,\theta)$
\begin{align*}
    && \mathbb E [H_{q-u_1}(\tilde{\beta}_j(x))H_{k_1}(\tilde\partial_1 \tilde{\beta}_j(x))H_{u_1-k_1}(\tilde\partial_2 \tilde{\beta}_j(x))H_{q-u_2}(\tilde{\beta}_j(y))H_{k_2}(\tilde\partial_1 \tilde{\beta}_j(y))H_{u_2-k_2}(\tilde\partial_2 \tilde{\beta}_j(y))]\\
    &&\qquad =\sum_{\alpha =q-k_1-u_2}^{\min (q-u_1, q-u_2)} M_{\alpha} \mathbb E [\tilde{\beta}_j(x)\tilde{\beta}_j(y)]^\alpha\mathbb E [\tilde{\beta}_j(x)\tilde\partial_1 \tilde{\beta}_j(y) ]^\beta \mathbb E [\tilde\partial_1 \tilde{\beta}_j(x)\tilde\partial_1 \tilde{\beta}_j(y)]^\gamma \mathbb E [\tilde\partial_2 \tilde{\beta}_j(x)\tilde\partial_2 \tilde{\beta}_j(y)]^{\delta},
\end{align*}
where $\beta = 2q-u_1-u_2-2\alpha$, $\gamma = k_1+u_2+\alpha -q$, $\delta = u_1-k_1$ and
\[M_{\alpha}=(q-u_1)!k_1!(u_1-k_1)!\binom{q-u_2}{ \alpha}\binom{k_2}{ q-u_1-\alpha}.\]
Moreover,
\[\sum_{\alpha =q-k_1-u_2}^{\min (q-u_1, q-u_2)} M_{\alpha} = (u_1-k_1)!(q-u_1+k_1)!\leq q!.\]
\end{lemma}
\begin{proof}
By diagram formula the term
$$ \mathbb E [H_{q-u_1}(\tilde{\beta}_j(x))H_{k_1}(\tilde\partial_1 \tilde{\beta}_j(x))H_{u_1-k_1}(\tilde\partial_2 \tilde{\beta}_j(x))H_{q-u_2}(\tilde{\beta}_j(y))H_{k_2}(\tilde\partial_1 \tilde{\beta}_j(y))H_{u_2-k_2}(\tilde\partial_2 \tilde{\beta}_j(y))]$$
equals
$$\sum_{\substack{\alpha,\beta,\gamma \geq 0\\
                  0\leq \alpha+\beta+\gamma \leq q}}M_{\alpha \beta\gamma} \mathbb E [\tilde{\beta}_j(x)\tilde{\beta}_j(y)]^\alpha\mathbb E [\tilde{\beta}_j(x)\tilde\partial_1 \tilde{\beta}_j(y) ]^\beta \mathbb E [\tilde\partial_1 \tilde{\beta}_j(x)\tilde\partial_1 \tilde{\beta}_j(y)]^\gamma \mathbb E [\tilde\partial_2 \tilde{\beta}_j(x)\tilde\partial_2 \tilde{\beta}_j(y)]^{q-\alpha-\beta-\gamma}.$$
 Note that, as shown in \cite{MRW}, other summands appearing in the diagram formula do not appear here, since
\[\mathbb E [\tilde{\beta}_j(x)\tilde\partial_2 \tilde{\beta}_j(y) ]=\mathbb E [\tilde\partial_1 \tilde{\beta}_j(x)\tilde\partial_2 \tilde{\beta}_j(y)]=0.\]
Let us determine the constant $M_{\alpha \beta\gamma}$ and understand how this expression can be simplified.\\

By diagram formula the number $M_{\alpha \beta\gamma}$ corresponds to the number of diagrams with $6$ rows with $q-u_1,\,k_1,\,u_1-k_1,\,q-u_2,\,k_2$ and $u_2-k_2$ vertices respectively, satisfying the following conditions:
\begin{enumerate}
\item[(i)] there are no horizontal connections,
\item[(ii)] there are $\alpha$ connections between the first and the fourth row,
\item[(iii)] there are $\beta$ connections between the first and the sixth ($\beta_1$ connections) and between the third and the fourth row ($\beta_2$ connections) together,
\item[(iv)] there are $\gamma$ connections between the third and the sixth row,
\item[(v)] there are $q-\alpha-\beta-\gamma$ connections between the second and the fifth row,
\item[(vi)] there are no further vertical connections.
\end{enumerate}
Since the second and fifth rows appear only in the condition (v), the number of vertices in these rows must be equal ($u_1-k_1=u_2-k_2$) and they must be connected to each other. The conditions (ii) and (iii) imply $\beta_1= q-u_1-\alpha$, $\beta_2=q-u_2-\alpha$, and with condition (iv) we conclude that $\gamma=k_1+u_2+\alpha-q=k_2+u_1+\alpha-q$. Moreover, $q-\alpha-\beta-\gamma=u_1-k_1=u_2-k_2$. In other words, all the powers in the formula are explicitly determined by $\alpha$, so the triple sum becomes a sum over $\alpha$. \\

Note that for such a diagram to exist $\alpha$ must lie between $\max(q-k_1-u_2,0)$ and $\min (q-u_1, q-u_2)$. Let us assume that $\min (q-u_1, q-u_2)=q-u_2$ without loss of generality. Let us now fix an $\alpha$ in this interval and calculate the number of possible diagrams for this $\alpha$. There are $\binom{q-u_1}{\alpha}$ and $\binom{q-u_2}{\alpha}$ possibilities to pick $\alpha$ vertices from the first and fourth rows respectively, there are $\alpha !$ possibilities to connect these vertices. There are $\binom{k_2}{q-u_1-\alpha}$ and $\binom{k_1}{q-u_2-\alpha}$ choices of $\beta_1$ and $\beta_2$ vertices respectively as well as $(q-u_1-\alpha)!$ and $(q-u_2-\alpha)!$ permutations for each. Additionally there are $(k_1-(q-u_2-\alpha))!$ connections of $\gamma$ and $(u_1-k_1)!$ connections of vertices between the second and the fifth rows. In total, we have
\begin{eqnarray*}
M_{\alpha \beta\gamma} &=& M_\alpha = \binom{q-u_1}{\alpha}\binom{q-u_2}{\alpha}\alpha !\binom{k_2}{q-u_1-\alpha}\binom{k_1}{q-u_2-\alpha}\\
&&\qquad \times(q-u_1-\alpha)!(q-u_2-\alpha)!(k_1-(q-u_2-\alpha))!(u_1-k_1)!\\
&=&(q-u_1)!k_1!(u_1-k_1)!\binom{q-u_2}{\alpha}\binom{k_2}{q-u_1-\alpha}.
\end{eqnarray*}
Let us compute $\sum_{\alpha}M_\alpha$.  Recall that this sum ranges from $q-k_1-u_2$ to $q-u_2$ (with the convention $\binom{n}{-k}=0$ for $k\in\mathbb N$). We obtain by shifting the index and using the identity $u_1-k_1=u_2-k_2$ 
\begin{eqnarray*}
\sum_{\alpha}M_\alpha &=&\sum_{\alpha =q-k_1-u_2}^{q-u_2} (q-u_1)!k_1!(u_1-k_1)!\binom{q-u_2}{\alpha}\binom{k_2}{q-u_1-\alpha}\\
&=& (q-u_1)!k_1!(u_1-k_1)!\sum_{\alpha =0}^{k_1}\binom{q-u_2}{\alpha + q-u_1-k_2}\binom{k_2}{k_2-\alpha}\\
&=& (q-u_1)!k_1!(u_1-k_1)! \sum_{\alpha =0}^{k_1} \binom{k_2}{\alpha}\binom{q-u_2}{k_1-\alpha}.
\end{eqnarray*}
By the Chu-Vandermonde identity the above sum over $\alpha$ equals $\binom{q-u_2+k_2}{k_1}$, and therefore,
$$\sum_{\alpha}M_\alpha = (q-u_1)!k_1!(u_1-k_1)! \binom{q-u_2+k_2}{k_1} =  (u_1-k_1)! (q-u_1+k_1)!\leq q!,$$
using, again, the identity $u_1-k_1=u_2-k_2$.
\end{proof}

The following proposition establishes a bound on asymptotic variance for all chaos components and answers the question of normalisation for the CLT.

\begin{prop}
    For $q\geq 3$ we have
    \[\mathbb E[\operatorname{proj}(\mathcal L_j(z)|C_q)^2]=O(1)\]
    with respect to $j$.
\end{prop}

\begin{proof}
We have
\begin{eqnarray*}
&&\mathbb E[\operatorname{proj}(\mathcal L_j(z)|C_q)^2]=A_{2^j}\sum_{u_1=0}^q\sum_{k_1=0}^{u_1}\sum_{u_2=0}^q\sum_{k_2=0}^{u_2}\frac{\alpha_{k_1,\,u_1-k_1}\beta_{q-u_1}(z)}{k_1!(u_1-k_1)!(q-u_1)!}\frac{\alpha_{k_2,\,u_2-k_2}\beta_{q-u_2}(z)}{k_2!(u_2-k_2)!(q_2-u_2)!}\\
&&\times \int_{\mathbb S^2}\int_{\mathbb S^2}\mathbb E [H_{q-u_1}(\tilde{\beta}_j(x))H_{k_1}(\tilde\partial_1 \tilde{\beta}_j(x))H_{u_1-k_1}(\tilde\partial_2 \tilde{\beta}_j(x))H_{q-u_2}(\tilde{\beta}_j(y))H_{k_2}(\tilde\partial_1 \tilde{\beta}_j(y))H_{u_2-k_2}(\tilde\partial_2 \tilde{\beta}_j(y))]dxdy\\
&&=A_{2^j}\sum_{u_1=0}^q\sum_{k_1=0}^{u_1}\sum_{u_2=0}^q\sum_{k_2=0}^{u_2}\frac{\alpha_{k_1,\,u_1-k_1}\beta_{q-u_1}(z)}{k_1!(u_1-k_1)!(q-u_1)!}\frac{\alpha_{k_2,\,u_2-k_2}\beta_{q-u_2}(z)}{k_2!(u_2-k_2)!(q_2-u_2)!}\\
&&\times \sum_{\alpha =q-k_1-u_2}^{\min (q-u_1, q-u_2)} M_{\alpha} \int_{\mathbb{S}^2\times \mathbb{S}^2}\mathbb E [\tilde{\beta}_j(x)\tilde{\beta}_j(y)]^\alpha\mathbb E [\tilde{\beta}_j(x)\tilde\partial_1 \tilde{\beta}_j(y) ]^\beta \mathbb E [\tilde\partial_1 \tilde{\beta}_j(x)\tilde\partial_1 \tilde{\beta}_j(y)]^\gamma \mathbb E [\tilde\partial_2 \tilde{\beta}_j(x)\tilde\partial_2 \tilde{\beta}_j(y)]^{\delta}\,dxdy
\end{eqnarray*}
Since $k_1$, $k_2$, $u_1$, $u_2$ are even (otherwise the factor in front of the integral equals zero), the powers $\beta = 2q-u_1-u_2-2\alpha$ and $\delta = u_1-k_1$ are even as well. This means that at least one of the integers $\alpha$, $\beta$, $\gamma$ and $\delta$, whose sum is $q$, is greater or equal than $2$. If $\alpha \geq 2$ we bound $|\mathbb E [\tilde{\beta}_j(x)\tilde\partial_1 \tilde{\beta}_j(y) ]|^\beta$, $|\mathbb E [\tilde\partial_1 \tilde{\beta}_j(x)\tilde\partial_1 \tilde{\beta}_j(y)]|^\gamma$ and $|\mathbb E [\tilde\partial_2 \tilde{\beta}_j(x)\tilde\partial_2 \tilde{\beta}_j(y)]|^{\delta}$ by $1$ (recalling that all these covariances are bounded by $1$ due to normalisation) and use
\[|\mathbb E [\tilde{\beta}_j(x)\tilde{\beta}_j(y)]|^\alpha\leq \mathbb E [\tilde{\beta}_j(x)\tilde{\beta}_j(y)]^2\]
to obtain
\begin{eqnarray*}
&& \left|\int_{\mathbb{S}\times \mathbb{S}^2} \mathbb E [\tilde{\beta}_j(x)\tilde{\beta}_j(y)]^\alpha\mathbb E [\tilde{\beta}_j(x)\tilde\partial_1 \tilde{\beta}_j(y) ]^\beta \mathbb E [\tilde\partial_1 \tilde{\beta}_j(x)\tilde\partial_1 \tilde{\beta}_j(y)]^\gamma \mathbb E [\tilde\partial_2 \tilde{\beta}_j(x)\tilde\partial_2 \tilde{\beta}_j(y)]^{\delta}\,dxdy\right|\\
&&\leq \int_{\mathbb{S}\times \mathbb{S}^2} \mathbb E [\tilde{\beta}_j(x)\tilde{\beta}_j(y)]^2 \,dxdy.
\end{eqnarray*}
This is bounded (up to a constant) by $S_1$ from Theorem \ref{2ndchaos}. Similarly, if $\beta$, $\gamma$ or $\delta$ is greater than $1$ the integral is bounded by $S_2$, $S_3$ or $S_4$ respectively. By Theorem \ref{2ndchaos} all of these yield summands of order $O(1)$.
\end{proof}

Let us now turn to the asymptotics for the case that the integrands are close to each other. This proof can be found in the appendix.

\begin{lemma}\label{smallint}
The integral
\[\lim_{\varepsilon_1,\varepsilon_2\to 0}\frac{1}{4\varepsilon_1\varepsilon_2}\int_0^{C/\ell}\mathbb E \left[ 1_{[z-\varepsilon_1, z+\varepsilon_1]}(\tilde\beta_j(x))\|\nabla\tilde{\beta}_j(x)\|1_{[z-\varepsilon_2, z+\varepsilon_2]}(\tilde\beta_j(y))\|\nabla\tilde{\beta}_j(y)\|\right]\sin \theta d\theta ,\]
where $\ell=2^j$ behaves as $O(1)$ as $j$ tends to infinity.
\end{lemma}

The next two results establish the exact asymptotics of integrals involved in higher chaoses. The proof of the first proposition is technical and will also be carried out in the appendix.
\begin{prop}\label{chaosesq}
    For $q=3$ and $q>4$ the terms $\mathbb E[\operatorname{proj}(\mathcal L_j(z)|C_q)^2]$ converge as $j$ tends to infinity.
\end{prop}
\begin{prop}
   For $q=3$ and $q\geq 4$ we have: Assuming that the constant $c_q$ defined in \cite{CM} is nonzero, the limiting variance $\lim_{j\to\infty}\mathbb E[\operatorname{proj}(\mathcal L_j(z)|C_q)^2]$ is zero for at most $q$ values of $z$, i.e. these terms do not converge faster than the variance of the second chaos.
\end{prop}
\begin{proof}
   The limit of
   \begin{eqnarray*}
&&\mathbb E[\operatorname{proj}(\mathcal L_j(z)|C_q)^2]=A_{2^j}\sum_{u_1=0}^q\sum_{k_1=0}^{u_1}\sum_{u_2=0}^q\sum_{k_2=0}^{u_2}\frac{\alpha_{k_1,\,u_1-k_1}\beta_{q-u_1}(z)}{k_1!(u_1-k_1)!(q-u_1)!}\frac{\alpha_{k_2,\,u_2-k_2}\beta_{q-u_2}(z)}{k_2!(u_2-k_2)!(q_2-u_2)!}\\
&&\times \int_{\mathbb S^2}\int_{\mathbb S^2}\mathbb E [H_{q-u_1}(\tilde{\beta}_j(x))H_{k_1}(\tilde\partial_1 \tilde{\beta}_j(x))H_{u_1-k_1}(\tilde\partial_2 \tilde{\beta}_j(x))H_{q-u_2}(\tilde{\beta}_j(y))H_{k_2}(\tilde\partial_1 \tilde{\beta}_j(y))H_{u_2-k_2}(\tilde\partial_2 \tilde{\beta}_j(y))]dxdy
\end{eqnarray*}
is by definition of $\beta_{\cdot}$ equal to $\phi (z)^2$ multiplied by a polynomial in $z$ of order $2q$. Since it is nonnegative, the polynomial does not have any simple zeroes. The only case in which it has more than $q$ zeroes is when all its coefficients are equal to zero. However, we know that its leading coefficient is the factor in front of $\beta_q(z)^2$, contained in the summand with $u_1=u_2=k_1=k_2=0$. The factor in front of the integral in this summand is nonzero, and the integral itself is, up to a positive constant, exactly $c_q$ from \cite{CM}. Under the assumption of this proposition it is nonzero.
\end{proof}

\section{Central Limit Theorem}\label{SecCLT}
In this section we finally prove the Central Limit Theorem for the boundary length of excursion sets of needlet random fields. The proof follows the same lines as the CLT proof in \cite{CM} and exploits well-known results of the Stein-Malliavin method detailed in \cite{NP}, in particular Proposition \ref{Theorem5.1.3} (Theorem 5.1.3 \cite{NP}). To this purpose we will need some technical lemmas which can be found in Appendix \ref{OnCLT}. Fundamental is the computation of the mean and the variance. The latter was computed in the previous section where we showed that each chaotic component of the boundary length has finite asymptotic variance. 
However, the first step toward the CLT is the evaluation of the mean. In the following proposition we obtain the exact expression of $\mathbb{E}[\mathcal L_j(z)]$.
\begin{proposition}
For every $z\in\mathbb R$ we have
\[\mathbb E \mathcal L_j(z)=\sqrt{A_{2^j}}e^{-z^2\slash 2}2\pi .\]
\end{proposition}
\begin{proof}
$\mathbb E \mathcal L_j(z)$ is the limit of $\mathbb E \mathcal L_j^{\varepsilon}(z)$ for $\varepsilon\to 0$, i.e. the limit of
\[\sqrt{A_{2^j}}\frac{1}{2\varepsilon}\int_{\mathbb{S}^2}\mathbb E \left[1_{[z-\varepsilon,\,z+\varepsilon]}(\tilde{\beta}_j(x))\sqrt{(\tilde\partial_1 \tilde{\beta}_j(x))^2+(\tilde\partial_2 \tilde{\beta}_j(x))^2}\right] dx.\]
As showed in section \ref{SectionNormalization},
$\tilde{\beta}_j(x), \tilde\partial_1 \tilde{\beta}_j(x),\tilde\partial_2 \tilde{\beta}_j(x) $ are all standard normal random variables
for every $x\in\mathbb S^2$. Therefore,
\begin{eqnarray*}
&&\frac{1}{2\varepsilon}\mathbb E \left[1_{[z-\varepsilon,\,z+\varepsilon]}(\tilde{\beta}_j(x))\sqrt{(\tilde\partial_1 \tilde{\beta}_j(x))^2+(\tilde\partial_2 \tilde{\beta}_j(x))^2}\right]\\
&&\qquad =\frac{1}{2\varepsilon} \int_{\mathbb R} 1_{[z-\varepsilon,\,z+\varepsilon]}(x)\frac{1}{\sqrt{2\pi}}e^{-x^2\slash 2}dx\int_{\mathbb R^2}\sqrt{y^2+z^2}\frac{1}{2\pi}e^{-y^2\slash 2-z^2\slash 2}dydz\\
&&\qquad = \frac{1}{2\varepsilon}\frac{1}{(2\pi)^{3\slash 2}}\int_{z-\varepsilon}^{z+\varepsilon}e^{-x^2\slash 2}dx\sqrt{2}\pi^{3\slash 2}=\frac{1}{2} \frac{1}{2\varepsilon}\int_{z-\varepsilon}^{z+\varepsilon}e^{-x^2\slash 2}dx\stackrel{\varepsilon\to 0}{\to}\frac{1}{2} e^{-z^2\slash 2}.
\end{eqnarray*}
Then we obtain
\[\mathbb E \mathcal L_j(z)=\sqrt{A_{2^j}}\frac{1}{2} e^{-z^2\slash 2}|\mathbb S^2|=\sqrt{A_{2^j}}e^{-z^2\slash 2}2\pi .\]
\end{proof}
Let us denote now
$$\tilde{\mathcal L}_j(z):=\mathcal L_j(z)- \mathbb{E}[\mathcal L_j(z)]. $$ As mentioned above, before proving our main result we need some technical lemmas (see Appendix \ref{OnCLT} for the proofs). Let us first introduce the notation
$$C_{quk}:=\frac{\alpha_{k,\,u-k}\beta_{q-u}(z)}{k!(u-k)!(q-u)!}.$$
\begin{lemma}\label{Sym}
	 For integers $q$, $q'$ the following bounds hold:
\begin{eqnarray}\label{bound2}
&& \mathbb E [\langle D\operatorname{proj}(\mathcal L_j(z)|C_q),-DL^{-1}\operatorname{proj}(\mathcal L_j(z)|C_{q'})\rangle_{\mathcal H}^2] \leq C A_{2^j}^2(K'_M)^{q+q'}(2^j)^{(2-M)(q+q')} \sum_{r=1}^{q\wedge q'} \nonumber\\
&& \qquad \qquad \times q^2(r-1)!^2\binom{q-1}{ r-1}^2 \binom{q'-1}{ r-1}^2 (q+q'-2r)!\left(\sum_{u=0}^q\sum_{k=0}^{u} C_{quk}\right)^2 \left(\sum_{u=0}^{q'}\sum_{k=0}^{u} C_{q'uk}\right)^2,\nonumber \\
&& \operatorname{Var}(\langle D\operatorname{proj}(\mathcal L_j(z)|C_q),-DL^{-1}\operatorname{proj}(\mathcal L_j(z)|C_{q})\rangle_{\mathcal H})\leq C A_{2j}^2 (K'_M)^{2q} (2^j)^{(2-M)2q} \sum_{r=1}^{q-1}\nonumber \\
&&\qquad\qquad\times q^2(r-1)!^2\binom{q-1}{ r-1}^4 (2q-2r)!\left(\sum_{u=0}^q\sum_{k=0}^{u} C_{quk}\right)^4
\end{eqnarray}
for every $M\in\mathbb N$ with a constant $C_M$ depending on $M$.
\end{lemma}
\begin{remark}\label{rem1}
Using the estimates in the proof of Theorem 9 in \cite{CM}, we can write
\begin{eqnarray*}
&& \mathbb E [\langle D\operatorname{proj}(\mathcal L_j(z)|C_q),-DL^{-1}\operatorname{proj}(\mathcal L_j(z)|C_{q'})\rangle_{\mathcal H}^2]\leq C A_{2^j}^2 (K_M')^{q+q'}(2^j)^{(2-M)(q+q')} \\
&& \qquad \qquad \times q^2(q-1)!(q'-1)! 3^{q+q'-2}\left(\sum_{u=0}^q\sum_{k=0}^{u} C_{quk}\right)^2 \left(\sum_{u=0}^{q'}\sum_{k=0}^{u} C_{q'uk}\right)^2\\
&& \operatorname{Var}(\langle D\operatorname{proj}(\mathcal L_j(z)|C_q),-DL^{-1}\operatorname{proj}(\mathcal L_j(z)|C_{q})\rangle_{\mathcal H})\leq C A_{2^j}^2 (K_M')^{2q}(2^j)^{(2-M)2q} \\
&&\qquad\qquad\times q^2(q-1)!^2 3^{2q-2}\left(\sum_{u=0}^q\sum_{k=0}^{u} C_{quk}\right)^4.
\end{eqnarray*}
\end{remark}

\begin{lemma}\label{boundjq} We have that
\begin{equation}\label{2}
\left|\sum_{u=0}^{q} \sum_{k=0}^{u} \frac{\alpha_{k,u-k}}{k!(u-k)!} \frac{\beta_{q-u}(z)}{(q-u)!}\right| \leq C_z 2^{q}\frac{1}{\sqrt{(q-1)!}},
\end{equation}
where the constant $C_z$ depends only on the level $z$.
\end{lemma}

Keeping in mind the results for the variance obtained in the previous section, we prove the Central Limit Theorem following the same steps as in Theorem 9 in \cite{CM}. 

\begin{proof}[Proof of Theorem \ref{Th1}]
Let us start by introducing the following notation
$$\tilde{\mathcal L}_{j,N}(z):=\sum_{q=2}^{N} \operatorname{proj}(\mathcal L_j(z)|C_q),\qquad \sigma_N^2:=\frac{\operatorname{Var}(\tilde{\mathcal L}_{j,N}(z))}{\operatorname{Var}(\mathcal L_j(z))},\qquad N_N\sim N(0,\,1).$$
By triangle inequality we have
\begin{eqnarray}\label{wass}
d_W\left(\frac{\mathcal L_j(z)- \mathbb{E}[\mathcal L_j(z)]}{\sqrt{\operatorname{Var}(\mathcal L_j(z))}}, N\right)&\leq & d_W \left(\frac{\mathcal L_j(z)-\mathbb{E}[\mathcal L_j(z)]}{\sqrt{\operatorname{Var}(\mathcal L_j(z))}}, \frac{\tilde {\mathcal L}_{j,N}(z)}{\sqrt{\operatorname{Var}(\mathcal L_j(z))}}\right)\nonumber \\
&&\qquad +d_W\left(\frac{\tilde {\mathcal L}_{j,N}(z)}{\sqrt{\operatorname{Var}(\mathcal L_j(z))}}, N_N\right)+d_W (N_N,N).
\end{eqnarray}
For the first part, since the Wasserstein distance can be bounded by the $L^2$ norm, we get
\begin{eqnarray*}
&&d_W \left(\frac{\mathcal L_j(z)-\mathbb{E}[\mathcal L_j(z)]}{\sqrt{\operatorname{Var}(\mathcal L_j(z))}}, \frac{\tilde {\mathcal L}_{j,N}(z)}{\sqrt{\operatorname{Var}(\mathcal L_j(z))}}\right) \leq \bigg\{ \mathbb{E} \bigg[ \frac{\mathcal L_j(z)-\mathbb{E}[\mathcal L_j(z)]}{\sqrt{\operatorname{Var}(\mathcal L_j(z))}}- \frac{\tilde {\mathcal L}_{j,N}(z)}{\sqrt{\operatorname{Var}(\mathcal L_j(z))}})  \bigg]^2 \bigg\}^{1/2}\\
&=& \dfrac{1}{{\sqrt{\operatorname{Var}(\mathcal L_j(z))}}} \bigg\{ \mathbb{E}\bigg[ \int_{\mathbb{S}^2} \sum_{q=N+1}^{\infty} \sqrt{A_{2^j}}\sum_{u=0}^q\sum_{k=0}^u\frac{\alpha_{k,\,u-k}\beta_{q-u}(z)}{k!(u-k)!(q-u)!}\\
&&\times \int_{\mathbb 
{S}^2}H_{q-u}(\tilde{\beta}_j(x))H_k(\tilde\partial_1 \tilde{\beta}_j(x))H_{u-k}(\tilde\partial_2 \tilde{\beta}_j(x)) dx
  \bigg]^2   \bigg\}^{1/2}\\
  &=& \dfrac{1}{{\sqrt{\operatorname{Var}(\mathcal L_j(z))}}}  A_{2^j} \\
&& \times \sum_{q=N+1}^{\infty}\sum_{u_1=0}^q\sum_{k_1=0}^{u_1}\sum_{u_2=0}^q\sum_{k_2=0}^{u_2}\frac{\alpha_{k_1,\,u_1-k_1}\beta_{q-u_1}(z)}{k_1!(u_1-k_1)!(q-u_1)!}\frac{\alpha_{k_2,\,u_2-k_2}\beta_{q-u_2}(z)}{k_2!(u_2-k_2)!(q-u_2)!}I_{q,\,j}^{(u_1,\,k_1,\,u_2,\,k_2)},
\end{eqnarray*}
where

\begin{eqnarray*} 
&&I_{q,\,j}^{(u_1,\,k_1,\,u_2,\,k_2)} = 8\pi^2
\sum_{\alpha}M_{\alpha } \times\\&& \times \int_0^\pi \mathbb E [\tilde{\beta}_j(x)\tilde{\beta}_j(y)]^\alpha\mathbb E [\tilde{\beta}_j(x)\tilde\partial_1 \tilde{\beta}_j(y) ]^\beta \mathbb E [\tilde\partial_1 \tilde{\beta}_j(x)\tilde\partial_1 \tilde{\beta}_j(y)]^\gamma \mathbb E [\tilde\partial_2 \tilde{\beta}_j(x)\tilde\partial_2 \tilde{\beta}_j(y)]^{q-\alpha-\beta-\gamma}\sin \theta d\theta 
\end{eqnarray*}
with the notation from Lemma \ref{constdiag}.
Now we note that the covariance functions in the integrand are 1 only at the origin, hence, we split the interval of the integral in $[0,C/\ell]$ and $[C/\ell, \pi]$. Following the same argument as in \cite{MR19} and \cite{DNPR} we can write that, in $[C/\ell, \pi]$, up to a constant we have
\begin{eqnarray*}\label{eq1}
  &&\sum_{q=N+1}^{\infty} \sum_{u_1=0}^q\sum_{k_1=0}^{u_1}\sum_{u_2=0}^q\sum_{k_2=0}^{u_2} \bigg|\frac{\alpha_{k_1,\,u_1-k_1}\beta_{q-u_1}(z)}{k_1!(u_1-k_1)!(q-u_1)!}\frac{\alpha_{k_2,\,u_2-k_2}\beta_{q-u_2}(z)}{k_2!(u_2-k_2)!(q-u_2)!} I_{q,\,j}^{(u_1,\,k_1,\,u_2,\,k_2)}|_{[\frac{C}{\ell},\pi]} \bigg| 
  \\ &\leq&  \sum_{q=N+1}^{\infty} \sum_{u_1=0}^q\sum_{k_1=0}^{u_1}\sum_{u_2=0}^q\sum_{k_2=0}^{u_2} \bigg|\frac{\alpha_{k_1,\,u_1-k_1}\beta_{q-u_1}(z)}{k_1!(u_1-k_1)!(q-u_1)!}\bigg| \bigg|\frac{\alpha_{k_2,\,u_2-k_2}\beta_{q-u_2}(z)}{k_2!(u_2-k_2)!(q-u_2)!}\bigg| \times q! (1-\delta)^{q}
  \end{eqnarray*}
for some $\delta >0$. Note that due to the covariance structure of the terms involved in the integrand (see \eqref{covbeta}) $\delta$ can be chosen to be close to one (this can be seen by means of Hilb's asymptotics and similar approximations, see Proposition \ref{P-asympt}). Now
\begin{eqnarray}\label{firstpart}
	&& \sum_{q=N+1}^{\infty} q! (1-\delta)^{q}  \sum_{u_1=0}^q\sum_{k_1=0}^{u_1}\sum_{u_2=0}^q\sum_{k_2=0}^{u_2} \bigg|\frac{\alpha_{k_1,\,u_1-k_1}\beta_{q-u_1}(z)}{k_1!(u_1-k_1)!(q-u_1)!}\bigg| \bigg|\frac{\alpha_{k_2,\,u_2-k_2}\beta_{q-u_2}(z)}{k_2!(u_2-k_2)!(q-u_2)!}\bigg| \nonumber \\
	&\leq&
	\sum_{q=N+1}^{\infty} q! (1-\delta)^{q} \sum_{2a+2b+c=q}\sum_{2a'+2b'+c'=q}\bigg|\frac{\alpha_{2a,\,2b}\beta_{c}(z)}{(2a)!(2b)!c!} \bigg| \bigg|\frac{\alpha_{2a',\,2b'}\beta_{c'}(z)}{(2a')!(2b')!c'!}\bigg|\nonumber\\
	&\leq & \sum_{\substack{a,b,c,a',b',c'\geq 0 \\2a+2b+c\geq N+1 \\ 2a'+2b'+c\geq N+1}} \bigg|\frac{\alpha_{2a,\,2b}\beta_{c}(z)}{(2a)!(2b)!c!} \bigg| \bigg|\frac{\alpha_{2a',\,2b'}\beta_{c'}(z)}{(2a')!(2b')!c'!}\bigg|
	\cdot \nonumber\\ && \cdot
	\sqrt{(2a+2b+c)!}  \sqrt{(2a'+2b'+c')!}\sqrt{1-\delta}^{2a+2b+c+2a'+2b'+c'}\nonumber\\
	&\leq & \sum_{\substack{a,b,c,a',b',c'\geq 0 \\2a+2b+c\geq N+1 \\ 2a'+2b'+c\geq N+1}}\bigg|\frac{\alpha_{2a,\,2b}\beta_{c}(z)}{(2a)!(2b)!c!} \bigg|^2 (2a+2b+c)!\sqrt{1-\delta}^{2a+2b+c+2a'+2b'+c'}.	
\end{eqnarray}
The last step is due to Cauchy-Schwarz inequality applied as in \cite{DNPR}.
The map $(a,b,c) \to \frac{\alpha_{2a,2b}^2 \beta_{c}(z)^2}{(2a)!(2b)!(c)!}$ is bounded, therefore, \eqref{firstpart} is bounded by
\begin{eqnarray*}
	&&	C \sum_{\substack{a,b,c,a',b',c'\geq 0 \\2a+2b+c\geq N+1 \\ 2a'+2b'+c'\geq N+1}}\frac{(2a+2b+c)!}{(2a)!(2b)!c!}\sqrt{1-\delta}^{2a+2b+c+2a'+2b'+c'}.
\end{eqnarray*}
Moreover, we have the estimate $\frac{(2a+2b+c)!}{(2a)!(2b)!c!}\leq 3^{2a+2b+c}$.\\

Now we consider the sums $\sum_{\substack{a,b,c \geq 0 \\2a+2b+c\geq N}}x^{2a+2b+c}$ for some $0< x <1$. We have that
\begin{eqnarray*}
	 &&\sum_{\substack{a,b,c \geq 0 \\2a+2b+c\geq N}}x^{2a+2b+c} = \sum_{c\geq 0}x^c\sum_{b\geq 0}x^{2b}\sum_{a=0\wedge \lceil(N-2b-c)/2\rceil}^{\infty}x^{2a}\\
	 &=& \sum_{c=0}^N x^c \sum_{b=0}^{(N-c)/2} x^{2b} \sum_{a=\lceil(N-2b-c)/2\rceil}^{\infty}x^{2a}+\sum_{c\geq 0}x^c \sum_{b=0\wedge \lceil(N-c)/2\rceil}x^{2b}\sum_{a\geq 0}x^{2a}\\
	 &=&\sum_{c=0}^N x^c \sum_{b=0}^{\lceil(N-c)/2\rceil} x^{2b} \sum_{a=\lceil(N-2b-c)/2\rceil}^{\infty}x^{2a} + \sum_{a\geq 0}x^{2a} \left(\sum_{c=0}^N x^c\sum_{b=\lceil(N-c)/2\rceil}^{\infty} x^{2b}+ \sum_{c=N}^\infty x^c\sum_{b=0}^{\infty} x^{2b}\right).
\end{eqnarray*}
The first summand is (up to a constant independent of $N$) of order $N^2x^{N-1}$, the second one of order $Nx^{N-1}$ and the last one is of order $x^{N-1}$. In total, the sum is of order $N^2 x^{N-1}$. Applying this to the sums in \eqref{firstpart} for $x= 3\sqrt{1-\delta}$ and $\sqrt{1-\delta}$ respectively we obtain the bound (up to a constant) $A_{2^j}N^4 (3(1-\delta))^N$. Now for $\theta \in [0, C/\ell]$, in light of Lemma \ref{smallint}, we have that

\begin{eqnarray*}
&&\sum_{q=1}^{\infty} \sqrt{A_{2^j}}\sum_{u_1=0}^q\sum_{k_1=0}^{u_1}\sum_{u_2=0}^q\sum_{k_2=0}^{u_2}\frac{\alpha_{k_1,\,u_1-k_1}\beta_{q-u_1}(z)}{k_1!(u_1-k_1)!(q-u_1)!}\frac{\alpha_{k_2,\,u_2-k_2}\beta_{q-u_2}(z)}{k_2!(u_2-k_2)!(q-u_2)!} 8\pi^2 \times \\&&
\sum_{\alpha} M_{\alpha } \int_0^{C/\ell}\mathbb E [\tilde{\beta}_j(x)\tilde{\beta}_j(y)]^\alpha\mathbb E [\tilde{\beta}_j(x)\tilde\partial_1 \tilde{\beta}_j(y) ]^\beta \mathbb E [\tilde\partial_1 \tilde{\beta}_j(x)\tilde\partial_1 \tilde{\beta}_j(y)]^\gamma \mathbb E [\tilde\partial_2 \tilde{\beta}_j(x)\tilde\partial_2 \tilde{\beta}_j(y)]^{q-\alpha-\beta-\gamma}\sin \theta d\theta
\end{eqnarray*}
converges and hence the tail of the series (q = $N+1, \dots, \infty$) goes to zero.\\

For the second summand on the right hand side of (\ref{wass}) we procede similarly to \cite{CM}. By Proposition \ref{Theorem5.1.3} we have
$$d_W\left(\frac{\tilde {\mathcal L}_{j,N}(z)}{\sqrt{\operatorname{Var}(\mathcal L_j(z))}}, N_N\right)\leq \frac{\sqrt{2}}{\sigma_N \sqrt{\pi}\operatorname{Var}(\mathcal L_j(z))} \mathbb E \left[\left| \operatorname{Var}(\tilde{\mathcal L}_{j,N}(z)) - \langle D \tilde{\mathcal L}_{j,N}(z), -DL^{-1}\tilde{\mathcal L}_{j,N}(z) \rangle\right|\right].$$
We know that the denominator is asymptotically constant in both $j$ and $N$. Moreover, since the Wiener chaos decomposition is orthogonal, we obtain
$$ \operatorname{Var}(\tilde{\mathcal L}_{j,N}(z))=\sum_{q=2}^{N} \mathbb E [\operatorname{proj}(\mathcal L_j(z)|C_q)^2]. $$
Therefore, we have
\begin{eqnarray*}
&&\mathbb E \left[\left| \operatorname{Var}(\tilde{\mathcal L}_{j,N}(z)) - \langle D \tilde{\mathcal L}_{j,N}(z), -DL^{-1}\tilde{\mathcal L}_{j,N}(z) \rangle\right|\right]\\
&&\qquad = \mathbb E\left[\left|\sum_{q=2}^{N} \mathbb E [\operatorname{proj}(\mathcal L_j(z)|C_q)^2] -\sum_{q=2}^{N}\sum_{q'=2}^{N}  \langle D\operatorname{proj}(\mathcal L_j(z)|C_q)^2, -DL^{-1}  \operatorname{proj}(\mathcal L_j(z)|C_{q'})^2\rangle \right|\right]\\
&&\qquad \leq \sum_{q=2}^{N} \mathbb E\left[\left| \mathbb E [\operatorname{proj}(\mathcal L_j(z)|C_q)^2] -\sum_{q'=2}^{N}  \langle D\operatorname{proj}(\mathcal L_j(z)|C_q)^2, -DL^{-1}  \operatorname{proj}(\mathcal L_j(z)|C_{q'})^2\rangle \right|\right]\\
&&\qquad \leq \sum_{q=2}^{N} \mathbb E\left[\left| \mathbb E [\operatorname{proj}(\mathcal L_j(z)|C_q)^2] - \langle D\operatorname{proj}(\mathcal L_j(z)|C_q)^2, -DL^{-1}  \operatorname{proj}(\mathcal L_j(z)|C_{q})^2\rangle \right|\right]\\
&&\qquad\qquad +\sum_{q=2}^{N} \sum_{q'=2, q\neq q'}^{N} \mathbb E\left[\left| \langle D\operatorname{proj}(\mathcal L_j(z)|C_q)^2, -DL^{-1}  \operatorname{proj}(\mathcal L_j(z)|C_{q'})^2\rangle \right|\right]\\ 
&&\qquad \leq \sum_{q=2}^{N} \operatorname{Var}\left( \langle D\operatorname{proj}(\mathcal L_j(z)|C_q)^2, -DL^{-1}  \operatorname{proj}(\mathcal L_j(z)|C_{q})^2\rangle\right)^{1\slash 2}\\
&&\qquad\qquad +\sum_{q=2}^{N} \sum_{q'=2, q\neq q'}^{N} \mathbb E\left[\langle D\operatorname{proj}(\mathcal L_j(z)|C_q)^2, -DL^{-1}  \operatorname{proj}(\mathcal L_j(z)|C_{q'})^2\rangle ^2\right]^{1\slash 2},
\end{eqnarray*}
where the last step follows by Theorem 2.9.1 in \cite{NP} applied to the first summand and Cauchy-Schwarz inequality for the second one.
By Remark \ref{rem1} we can now bound this by
\begin{eqnarray*}
&&\sum_{q=2}^{N} \sqrt{C} A_{2^j}(K_M')^{q}  (2^j)^{(2-M)q}  q (q-1)! 3^{q-1}\left(\sum_{u=0}^q\sum_{k=0}^{u} C_{quk}\right)^2\\
&&\qquad +\sum_{q=2}^{N} \sum_{q'=2, q\neq q'}^{N} \sqrt{C} A_{2^j} \sqrt{K_M'}^{q+q'} (2^j)^{((2-M)(q+q'))\slash2} q \sqrt{(q-1)!(q'-1)!} \\ &&\qquad \qquad \times 3^{(q+q'-2)\slash 2}\left|\sum_{u=0}^q\sum_{k=0}^{u} C_{quk}\sum_{u=0}^{q'}\sum_{k=0}^{u} C_{q'uk}\right|.
\end{eqnarray*}
Now we can use the result (\ref{2}) and bound this further by
\begin{eqnarray*}
\frac{A_{2^j}\sqrt{C}C_z^2 }{3}&&\bigg(\sum_{q=2}^{N} (K'_M)^{q}(2^j)^{(2-M)q}  q 2^q 3^{q} \\&&+\sum_{q=2}^{N} \sum_{q'=2, q\neq q'}^{N} (K'_M)^{(q+q')/2}(2^j)^{((2-M)(q+q'))/2} q  3^{(q+q')\slash 2}2^{q+q'}\bigg)\\
\end{eqnarray*}
which for big $j$ and $M>4$ is
$$\lesssim CC_z 2^{2j} 2^{j(2-M)}\left[ N(K_M'6)^{N} 2^{j(2-M)N}+(K_M'6)^22^{2j(2-M)}+(K_M'6)2^{j(2-M)} \right].$$\\

Finally for the third summand in (\ref{wass}) by Proposition 3.6.1 in \cite{NP}, we have that
\begin{eqnarray*}
d_W\left( N_N, N \right) &\leq& \sqrt{\frac{2}{\pi}} \dfrac{1}{\max \left(1, \sqrt{\frac{\operatorname{Var}(\tilde{\mathcal L}_{j,N}(z))}{\operatorname{Var}(\mathcal L_{j}(z))}}\right)} \left| 1- \frac{\operatorname{Var}(\tilde{\mathcal L}_{j,N}(z))}{\operatorname{Var}(\mathcal L_{j}(z))} \right|\\&=&\sqrt{\frac{2}{\pi}} 
\mathbb{E} \bigg[ \frac{\mathcal L_j(z)-\mathbb{E}[\mathcal L_j(z)]}{\sqrt{\operatorname{Var}(\mathcal L_j(z))}}- \frac{\tilde {\mathcal L}_{j,N}(z)}{\sqrt{\operatorname{Var}(\mathcal L_j(z))}}  \bigg]^2
\end{eqnarray*}
which, as shown in the first part of the proof, is bounded by $C_2 A_{2^j}^2 [N^{4}(3(1-\delta))^N]^2$. \\
Now, we put together the bounds we obtained for the three terms of the right hand side of (\ref{wass}) and choosing $N:=N(j)=j$, $\delta>\frac{11}{12}$ and $M>4$ we have that all the three terms go to zero as $j \to \infty$  and the thesis of the theorem follows.
\end{proof}

\section{Appendix}

\subsection{Properties of needlet systems}\label{Appkernel}
In this section we recall some analytic properties of the needlet
systems, in particular their localization properties in real and harmonic spaces (see Theorem 3.5 \cite{NPW}). The latter allows to show that needlet coefficients
are asymptotically uncorrelated for any fixed angular distance, as the
frequency goes to infinity. \\Let us consider the kernel
\begin{eqnarray*}
\Psi_j^{(1)}(\langle x, y \rangle)&:=& \sum_\ell b\left(\frac{\ell}{B^j}\right) \frac{2\ell+1}{4\pi} P_\ell(\langle x,y \rangle).
\end{eqnarray*}
In Theorem 13.1 \cite{MP} (see also \cite{BKMP09} and Proposition 10.5 \cite{MP}) the authors showed that for all $M \in \mathbb{N}$ there exists a positive constant $C_M'$ such that
\begin{equation}\label{loc0}
    |\Psi_j^{(1)}(\langle x,y\rangle)| \leq \frac{C^{'}_M B^{2j}}{(1+B^j d(x,y))^M}.
\end{equation}
Moreover, since it is known that the derivative of an operator kernel related to the Laplace-Beltrami operator on the sphere is a kernel itself (see equation (12) in \cite{NPW}) we can establish the same property on the kernels $\frac{\partial}{\partial \langle x,y \rangle} \Psi_j^{(1)}(\langle x, y \rangle)$ and $\frac{\partial^2}{\partial \langle x,y \rangle^2} \Psi_j^{(1)}(\langle x, y \rangle)$ applying Lemma 2.1 \cite{S} (see also Corollary 5.3 \cite{NPW}). Hence, we get 
\begin{eqnarray}\label{kernel1}
\bigg|\frac{\partial}{\partial \langle x,y \rangle}\Psi_j^{(1)}(\langle x,y\rangle)\bigg| &\leq& \frac{C^{''}_M B^{4j}}{(1+B^j d(x,y))^M},
\end{eqnarray}
\begin{eqnarray}\label{kernel2}
\bigg|\frac{\partial^2}{\partial \langle x,y \rangle^2}\Psi_j^{(1)}(\langle x,y\rangle)\bigg| &\leq& \frac{C^{'''}_M B^{6j}}{(1+B^j d(x,y))^M},
\end{eqnarray}
where $C_M'', C_M'''$ are positive constants.\\

Now we define
\begin{eqnarray*}
\Psi_j^{(2)}(\langle x, y \rangle)&:=& \sum_\ell b\left(\frac{\ell}{B^j}\right) \frac{2\ell+1}{4\pi} P_\ell^\prime(\cos d(x,y)) \sin d(x,y)= - \frac{\partial}{\partial (\langle x,y \rangle)} \Psi_j^{(1)}(\langle x, y \rangle) \cdot \sin (d(x, y)),\\
	\Psi_j^{(3)}(\langle x, y \rangle)&:=& \sum_\ell b\left(\frac{\ell}{B^j}\right) \frac{2\ell+1}{4\pi} (-P_\ell^\prime(\cos d(x,y)) \cos (d(x,y)) +P_\ell^{''} (\cos d(x,y)) \sin^2 d(x,y)) \\&=& \frac{\partial^2}{\partial (\langle x,y\rangle )^2} \Psi_j^{(1)}(\langle x, y \rangle)  \sin (d(x, y))^2- \frac{\partial}{\partial (\langle x,y\rangle )} \Psi_j^{(1)}(\langle x, y \rangle) \cos (d(x, y))^2, \\
	\Psi_j^{(4)}(\langle x, y \rangle)&:=& \sum_\ell b\left(\frac{\ell}{B^j}\right) \frac{2\ell+1}{4\pi} P_\ell^\prime(\cos d(x,y))= -\frac{\partial}{\partial (\langle x,y \rangle)} \Psi_j^{(1)}(\langle x, y \rangle).
\end{eqnarray*}
Then exploiting \eqref{kernel1} and \eqref{kernel2} we conclude that for all $M \in \mathbb{N}$ there exist positive costants $C^{'}_M,C^{''}_M,C^{'''}_M,C^{''''}_M $ such that
\begin{eqnarray}\label{loc1}
|\Psi_j^{(2)}(\langle x,y\rangle)| &\leq& \frac{C^{''}_M B^{4j}}{(1+B^j d(x,y))^M},\\
|\Psi_j^{(3)}(\langle x,y\rangle)| &\leq& \frac{C^{'''}_M B^{6j}}{(1+B^j d(x,y))^M},\\
|\Psi_j^{(4)}(\langle x,y\rangle)| &\leq& \frac{C^{''''}_MB^{4j}}{(1+B^j d(x,y))^M}.
\end{eqnarray}
Under Condition \ref{cond1} the localization property in \eqref{loc0} allows to find an upper bound for the correlation coefficients of $\{\beta_{j}(\cdot)\}$ (Lemma 10.8 \cite{CM} and \cite{BKMP09}). In the following proposition we show that similar results can be derived also for $Corr(\beta_j(x), \partial_1\beta_{j}(y)),$ $ Corr(\partial_1 \beta_j(x),$ $ \partial_1\beta_{j}(y))$ and  $Corr(\partial_2 \beta_j(x), \partial_2\beta_{j}(y))$.
\begin{prop}\label{boundCov}
Under Condition \ref{cond1}, for all $M \in \mathbb{N}$, there exist positive constants $C_M', C_M'',C_M''',C_M'''' $ such that the following inequalities hold:
	\begin{eqnarray}\label{corr1}
|\tilde\rho_1(x,y)|&=&|Corr(\beta_j(x), \beta_{j}(y))| \leq \frac{C_M'}{(1+2^jd(x,y))^M},
\\
\label{corr2}
|\tilde\rho_2(x,y)|&=&|Corr(\beta_j(x), \partial_1\beta_{j}(y))| \leq \frac{C_M^{''} 2^{j}}{(1+2^jd(x,y))^M}, 
\\
|\tilde\rho_3(x,y)|&=&|Corr(\partial_1 \beta_j(x), \partial_1\beta_{j}(y))| \leq \frac{C_M^{'''} 2^{2j}}{(1+2^jd(x,y))^M},
\\
|\tilde\rho_4(x,y)|&=&|Corr(\partial_2 \beta_j(x), \partial_2\beta_{j}(y))| \leq \frac{C^{''''}_M}{(1+2^jd(x,y))^M}.
\end{eqnarray}
\end{prop}

\begin{proof}
The proof of all inequalities follows the same steps as the one in Lemma 10.8 \cite{MP}; where (\ref{corr1}) has been proved. In order to get (\ref{corr2}) we have to bound	$$Corr(\beta_j(x), \partial_1\beta_{j}(y))=\dfrac{\mathbb{E}[\beta_j(x) \partial_1\beta_{j}(y)]}{\sqrt{\mathbb{E}[\beta_j(x)^2] \mathbb{E}[\partial_1\beta_{j}(y)^2]}}= \dfrac{\sum_\ell b\left(\frac{\ell}{2^j}\right)^2 C_\ell \frac{2\ell+1}{4\pi} P_\ell^{\prime}(\cos d(x,y))\sin d(x,y)}{\sqrt{\sum_\ell b\left(\frac{\ell}{2^j}\right)^2 C_\ell \frac{2\ell+1}{4\pi} \sum_\ell b\left(\frac{\ell}{2^j}\right)^2 C_\ell \frac{\ell(\ell+1)}{2}\frac{2\ell+1}{4\pi}}}.$$
Note that condition \ref{cond1} implies that $c_1 \ell^{-\alpha} \leq C_\ell \leq c_2 \ell^{-\alpha}$, then we have
\begin{equation}\label{boundB}
c_1 2^{(2-\alpha)j} \leq\sum_\ell b\left(\frac{\ell}{2^j}\right)^2 C_\ell \frac{2\ell+1}{4\pi} \leq c_2 2^{(2-\alpha)j}
\end{equation}
and
\begin{equation}\label{boundA}
    c_1 2^{(4-\alpha)j} \leq\sum_\ell b\left(\frac{\ell}{2^j}\right)^2 C_\ell \frac{\ell(\ell+1)}{2}\frac{2\ell+1}{4\pi} \leq c_2 2^{(4-\alpha)j}.
\end{equation}

Now we define the sequence $\eta_j:=b^2(x)x^{-\alpha}g_j(x)$. As pointed out in \cite{MP}, $\eta_j$ satisfies the same conditions as $b(\cdot)$ which ensures the localization property. Hence (\ref{loc1}) implies that
$$\left|\sum_\ell  b\left(\frac{\ell}{2^j}\right)^2 \left(\frac{\ell}{2^j}\right)^{-\alpha} g_j \left(\frac{\ell}{2^j} \right)\frac{2\ell+1}{4\pi} P_\ell^{\prime}(\cos d(x,y))\sin d(x,y) \right| \leq \dfrac{C_M^{''}2^{4j}}{(1+2^j d(x,y))^M}$$
	This result together with (\ref{boundA}) and (\ref{boundB}) gives
	\begin{eqnarray*}
		|Corr(\beta_j(x), \partial_1\beta_{j}(y))| &=& \left| \dfrac{\sum_\ell b\left(\frac{\ell}{2^j}\right)^2  C_\ell (2\ell+1)P_\ell^{\prime}(\cos d(x,y))\sin d(x,y)}{\sqrt{\sum_\ell b\left(\frac{\ell}{2^j}\right)^2 C_\ell (2\ell+1) \sum_\ell b\left(\frac{\ell}{2^j}\right)^2 C_\ell \frac{\ell(\ell+1)}{2}(2\ell+1)}}\right|  \\&\leq&\left| 2^{-j\alpha} \dfrac{\sum_\ell b\left(\frac{\ell}{2^j}\right)^2 \ell^{-\alpha} g\left(\frac{\ell}{2^j}\right) (2\ell+1)P_\ell^{\prime}(\cos d(x,y))\sin d(x,y)}{2^{-j\alpha}\sqrt{\sum_\ell b\left(\frac{\ell}{2^j}\right)^2 C_\ell (2\ell+1) \sum_\ell b\left(\frac{\ell}{2^j}\right)^2 C_\ell \frac{\ell(\ell+1)}{2}(2\ell+1)}}\right| \\
		&\leq & \dfrac{C_M^{''}2^{4j}2^{-j\alpha}}{(1+2^j d(x,y))^M \sqrt{2^{(2-\alpha)j} 2^{(4-\alpha)j}}}= \dfrac{C_M^{''}2^{j}}{(1+2^j d(x,y))^M}
	\end{eqnarray*}
	Similarly we get
	\begin{eqnarray*}
		|Corr(\partial_1\beta_j(x), \partial_1\beta_{j}(y))| &=& \dfrac{\mathbb{E}[\partial_1\beta_j(x) \partial_1\beta_{j}(y)]}{\sqrt{\mathbb{E}[\partial_1\beta_j(x)^2] \mathbb{E}[\partial_1\beta_{j}(y)^2]}}\\
		&\leq & \dfrac{C_M^{'''}2^{6j}2^{-j\alpha}}{(1+2^j d(x,y))^M {2^{(4-\alpha)j} }}= \dfrac{C_M^{'''}2^{2j}}{(1+2^j d(x,y))^M}
	\end{eqnarray*}
	and 
	\begin{eqnarray*}
		|Corr(\partial_2\beta_j(x), \partial_2\beta_{j}(y))| &=& \dfrac{\mathbb{E}[\partial_2\beta_j(x) \partial_2\beta_{j}(y)]}{\sqrt{\mathbb{E}[\partial_2\beta_j(x)^2] \mathbb{E}[\partial_2\beta_{j}(y)^2]}}\\
		&\leq & \dfrac{C_M^{''''}2^{4j}2^{-j\alpha}}{(1+2^j d(x,y))^M {2^{(4-\alpha)j} }}= \dfrac{C_M^{'''}}{(1+2^j d(x,y))^M}.
	\end{eqnarray*}
\end{proof}

\subsection{Other auxiliary results}\label{SecAuxiliary}
A technical result used to simplify expectations of products of Hermite polynomials is the so called diagram formula (see \cite{MP} for details). Let us introduce the necessary notation.

For an integer $p\geq 1$ and $(\ell_1,\dots ,\ell_p)\in\mathbb N^p$ a diagram $\gamma$ of order $(\ell_1,\dots ,\ell_p)$ is a graph of $\ell_1$ vertices indexed by $1$, $\ell_2$ vertices indexed by $2$ and so forth until $p$ (this can be viewed as vertices ordered in $p$ rows) with each vertex having degree $1$. We denote by $\eta_{ik}(\gamma)$ the number of edges connecting vertices in rows $i$ and $k$ for a given diagram $\gamma$. We say that a diagram has no flat edges if there are no edges connecting vertices belonging to the same row. We denote by $\Gamma (\ell_1,\dots , \ell_p)$ the set of all diagrams of order $(\ell_1,\dots ,\ell_p)$ with no flat edges.

We can now formulate the diagram formula for moments.
\begin{prop}\label{diag}
Let $(Z_1,\dots ,Z_p)$ be a centred Gaussian vector and let $H_{\ell_1},\dots , H_{\ell_p}$ be Hermite polynomials of degrees $\ell_1,\dots, \ell_p \geq 1$ respectively. Then
\[\mathbb E \left[\prod_{j=1}^p H_{\ell_j}(Z_j)\right]=\sum_{\gamma\in \Gamma (\ell_1,\dots , \ell_p)}\prod_{1\leq i\leq j\leq p}\mathbb E [Z_iZ_j]^{\eta_{ij}(\gamma)}.\]
\end{prop}

The following proposition concerns Legendre polynomials.

\begin{prop}\label{deriv}
	For two Legendre polynomials $P_m$, $P_\ell$ we have the following identities:
	\begin{eqnarray*}
		\int_{-1}^1 P'_m(x)P'_\ell(x)dx=\begin{cases}
			& 0 \text{ if }m-\ell\equiv 1 \mod 2,\\ & 4\lceil\frac{\min(m,\ell)}{2}\rceil \text{ else}
		\end{cases}
	\end{eqnarray*}
	as well as
	\begin{eqnarray*}
		\int_{-1}^1 P_m(x)P'_\ell(x)dx=\begin{cases}
			& 0 \text{ if }m-\ell\equiv 0 \mod 2\text{ or }m\geq \ell,\\ & 2 \text{ else}.
		\end{cases}
	\end{eqnarray*}
\end{prop}
\begin{proof}
Combining the well-known formulae (see \cite{AS})
\[(n+1)P_{n+1}(x)=(2n+1)xP_n(x)-nP_{n-1}(x)\]
and
\[(1-x^2)P'_n(x)=-nxP_n(x)+nP_{n-1}(x),\]
we obtain the relation
\[(2n+1)P_n(x)=P'_{n+1}(x)-P'_{n-1}(x).\]
It follows that
\[P'_{n+1}(x)=\sum_{k=0}^{\lfloor n\slash 2\rfloor}(2(n-2k)+1)P_{n-2k}(x).\]
The identities now follow by orthogonality of Legendre polynomials.
\end{proof}
Let us finish the section by citing an asymptotic result of Hilb's type shown in \cite{W} concerning Legendre polynomials and their derivatives.
\begin{prop}\label{P-asympt}
For $0<C<\psi <\ell\frac{\pi}{2}$ the following expansions hold:
\begin{eqnarray*}
	P_n \left(\cos\frac{\psi}{\ell}\right)&=&\sqrt{\frac{2}{\pi n \sin\frac{\psi}{\ell}}}\left(\sin\left(\psi + \frac{\pi}{4}\right)+O\left(\frac{1}{\psi}\right)\right),\\
	P'_n\left(\cos\frac{\psi}{\ell}\right)&=&\sqrt{\frac{2}{\pi n \sin^3\frac{\psi}{\ell}}}\left(n \sin\left(\psi - \frac{\pi}{4}\right)+O\left(1\right)\right),\\
	P''_n\left(\cos\frac{\psi}{\ell}\right)&=&-\frac{n^2}{\sin^2 \frac{\psi}{\ell}}P_n\left(\cos\frac{\psi}{\ell}\right)+\frac{2}{\sin^2\frac{\psi}{\ell}}	P'_n\left(\cos\frac{\psi}{\ell}\right)+O\left(\frac{n^3}{\psi^{5/2}}\right).
\end{eqnarray*}
\end{prop}

\subsection{Proofs of Lemma \ref{smallint} and Proposition \ref{chaosesq}}\label{Prooflem3prop3}
\begin{proof}[Proof of Lemma \ref{smallint}]
    Let us introduce the following notation:
\[U(x):=\left(b\left(\frac{\ell}{2^j}\right)\sqrt{C_\ell}\frac{1}{\sqrt{B_{2^j}}} Y^\ast_{\ell m}(x)\right)_{\ell=2^{j-1},\dots , 2^{j+1},\,m=-\ell,\dots , \ell }\]
and
\[\tilde{\beta}_j:=\left(\frac{a^\ast_{\ell m }}{\sqrt{C_\ell}}\right)_{\ell=2^{j-1},\dots , 2^{j+1},\,m=-\ell,\dots , \ell },\]
where  $\frac{a^\ast_{\ell m}}{\sqrt{C_\ell}}\sim N(0,1)$ are real i.i.d.random variables and $(Y^\ast_{\ell m})_{m=-\ell ,\dots ,\ell}$ is the appropriately scaled real basis of the $\ell$th eigenspace (see \cite{MP} for the precise relation between $Y^\ast_{\ell m}$ and $Y_{\ell m}$ and the respective coefficients $a^\ast_{\ell m}$ and $a_{\ell m}$) such that
\[\langle U(x),\tilde\beta_j\rangle =\tilde\beta_j(x).\]
Let $n$ be the number of components of the vectors defined above.

As $\tilde{\beta}_j$ is a vector of i.i.d. standard normal random variables, we can also introduce the measure associated with it, namely
\[d\nu_{\tilde\beta}:=e^{-\frac{\|\tilde\beta\|^2}{2}}\frac{1}{(2\pi)^{n\slash 2}}d\tilde\beta_1\dots d\tilde\beta_n.\]
Recalling that
\[\mathbb E [ \mathcal L_j(z)^2]=\lim_{\varepsilon_1,\varepsilon_2\to 0}\frac{1}{4\varepsilon_1\varepsilon_2}\int_{\mathbb{S}^2}\int_{\mathbb{S}^2}\mathbb E\left[ 1_{[z-\varepsilon_1, z+\varepsilon_2]}(\tilde\beta_j(x))\|\nabla\tilde{\beta}_j(x)\|1_{[z-\varepsilon_2, z+\varepsilon_2]}(\tilde\beta_j(y))\|\nabla\tilde{\beta}_j(y)\|\right] dxdy,\]
with the new notation we can write the integrand above as
\[\frac{1}{4\varepsilon_1\varepsilon_2}\int_{\mathbb R^n}1_{[z-\varepsilon, z+\varepsilon]}(\tilde\beta_j(x))\|\nabla\tilde{\beta}_j(x)\|1_{[z-\varepsilon, z+\varepsilon]}(\tilde\beta_j(y))\|\nabla\tilde{\beta}_j(y)\|d\nu_{\tilde\beta}=:K_{\varepsilon_1\varepsilon_2}(x,y).\]
Write, moreover $K(x,y):=\lim_{\varepsilon_1,\,\varepsilon_2\to 0}K_{\varepsilon_1\varepsilon_2}(x,y)$. Note that by isotropy we have
\[\mathbb E \mathcal L_j(z)^2 \sim \int_0^{\pi} K(x,N) \sin \theta d\theta , \]
where $N=(0,0)$ and $x=(\theta ,0)$, and the integral that we want to study is
\[ \sim \int_0^{C\slash \ell} K(x,N) \sin \theta d\theta .\]

With the new definitions given above we have
\[\langle U(x),U(y)\rangle = \mathbb E [\tilde\beta_j(x)\tilde\beta_j(y)]=:u(x,y),\]
\[\|\partial_iU(x)\|^2 = \mathbb E [\partial_i\tilde\beta_j(x)^2]=A_{2^j}.\]

Let us write using the definition of $d\nu_{\tilde\beta}$
\[K_{\varepsilon_1\varepsilon_2}(x,y)=\frac{1}{4\varepsilon_1\varepsilon_2}\frac{1}{(2\pi)^{n\slash 2}}\int_{\substack{|\tilde\beta_j(x)-z|<\varepsilon_1 \\ |\tilde\beta_j(y)-z|<\varepsilon_2 }} \|\nabla\tilde{\beta}_j(x)\|\|\nabla\tilde{\beta}_j(y)\|e^{-\|\tilde\beta_j\|^2\slash 2}d\tilde\beta_j.\]
We have: $\tilde\beta_j(x)=\langle U(x),\tilde\beta_j\rangle = \cos \alpha_{U(x),\tilde\beta_j}\|U(x)\|\|\tilde \beta_j\|$, which is, up to the sign, equal to the length of the projection of $\tilde\beta_j$ on the line generated by $U(x)$. Let us denote by $\pi$ the plain spanned by $U(x)$ and $U(y)$. It follows that the above domain of integration are vectors whose projections onto $\pi$ lie in one of the red parallelograms in Figure \ref{fig:1} (note that for $z=0$ these parallelograms coincide).
\begin{figure}[ht!]
\centering
\includegraphics[scale=0.75]{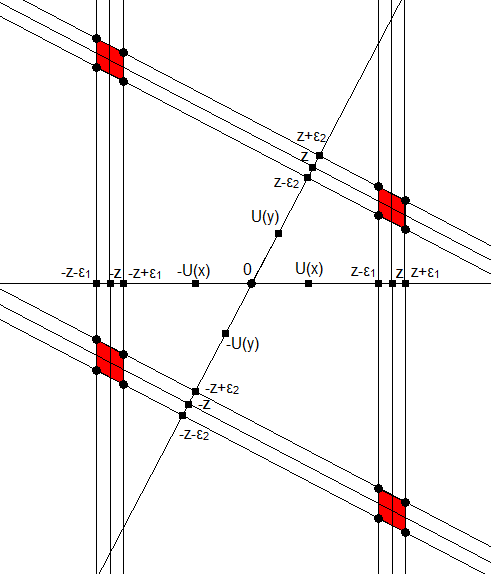}
\caption{Domain of integration}
\label{fig:1}
\end{figure}
Each of these parallelograms is of the size of the parallelogram $P$ in Lemma 3.4 in \cite{W1}, namely $4\varepsilon_1\varepsilon_2 \frac{1}{\sqrt{1-u(x,y)^2}}$, where $u(x,y)$ is the quantity defined above (and different from $u(x,y)$ in \cite{W1}). We call the figure comprised of the four parallelograms $P_z$ and rewrite as in \cite{W1}:
\begin{eqnarray*}
&&\frac{1}{(2\pi)^{n\slash 2}}\int_{\substack{|\tilde\beta_j(x)-z|<\varepsilon_1 \\ |\tilde\beta_j(y)-z|<\varepsilon_2 }} \|\nabla\tilde{\beta}_j(x)\|\|\nabla\tilde{\beta}_j(y)\|e^{-\|\tilde\beta_j\|^2\slash 2}d\tilde\beta_j\\
&&=\frac{1}{(2\pi)^{n\slash 2}}\int_{\substack{|\tilde\beta_j(x)-z|<\varepsilon_1 \\ |\tilde\beta_j(y)-z|<\varepsilon_2 }} \sqrt{\left(\langle\tilde\beta_j , \partial_1 U(x)\rangle^2+\langle\tilde\beta_j , \partial_2 U(x)\rangle^2\right)\left(\langle\tilde\beta_j , \partial_1 U(y)\rangle^2+ \langle\tilde\beta_j , \partial_2 U(y)\rangle^2\right)} e^{-\|\tilde\beta_j\|^2\slash 2}d\tilde\beta_j\\
&& =\frac{1}{(2\pi)^{n\slash 2}}\int_{P_z}\int_{p+\pi^{\perp}}\sqrt{\left(\langle f , \partial_1 U(x)\rangle^2+\langle f , \partial_2 U(x)\rangle^2\right)\left(\langle f , \partial_1 U(y)\rangle^2+ \langle f , \partial_2 U(y)\rangle^2\right)}e^{-\|f\|^2\slash 2}df dp,
\end{eqnarray*}
using
\[\nabla\tilde\beta_j(x)=\left(\langle\tilde\beta_j , \partial_1 U(x)\rangle, \langle\tilde\beta_j , \partial_2 U(x)\rangle \right)^T.\]
Writing out the inner products for $f=(f_i)_{i=1\dots n}$ we obtain
\[\langle f , \partial_1 U(x)\rangle^2+\langle f , \partial_2 U(x)\rangle^2 =\sum_{i,j=1}^n f_if_j( (\partial_1U(x))_i(\partial_1U(x))_j +  (\partial_2U(x))_i(\partial_2U(x))_j)=f^T\mathcal U(x) f\]
with a symmetric positive definite matrix
\[\mathcal U(x)=( (\partial_1U(x))_i(\partial_1U(x))_j +  (\partial_2U(x))_i(\partial_2U(x))_j)_{i,j=1,\dots ,n}= (\partial_1U(x)) (\partial_1U(x))^T+  (\partial_2U(x)) (\partial_2U(x))^T,\]
such that the above integral can be written as
\[\frac{1}{(2\pi)^{n\slash 2}}\int_{P_z}\int_{p+\pi^{\perp}}\sqrt{f^T\mathcal U(x) f}\sqrt{f^T\mathcal U(y) f}e^{-\|f\|^2\slash 2}df dp.\]
Note that for $f_1\in\pi^{\perp}$ we have by triangle inequality for the norms induced by $\mathcal U(x)$ and $\mathcal U(y)$
\begin{eqnarray*}
&&\sqrt{(p+f_1)^T\mathcal U(x) (p+f_1)}\sqrt{(p+f_1)^T\mathcal U(y) (p+f_1)} e^{-\|p+f_1\|^2\slash 2}\\
&&\qquad \lesssim \left(\sqrt{\operatorname{tr}(\mathcal U(x))}+\sqrt{f_1^T\mathcal U(x) f_1}\right)\left(\sqrt{\operatorname{tr}(\mathcal U(y))}+\sqrt{f_1^T\mathcal U(y) f_1}\right) e^{-\|f_1\|^2\slash 2},
\end{eqnarray*}
since $p^T\mathcal U(x) p$ is bounded by $\operatorname{tr}(\mathcal U(x))$. Thus, under an appropriate rotation $O$, the inner integral can be bounded up to an absolute constant by
\begin{eqnarray*}
&&\frac{1}{(2\pi)^{(n-2)\slash 2}} \int_{\mathbb R^{n-2}} \left(\sqrt{\operatorname{tr}(\mathcal U(x))}+\sqrt{f^T\mathcal O^TU(x)O f}\right)\left(\sqrt{\operatorname{tr}(\mathcal U(y))}+\sqrt{f^TO^T\mathcal U(y)O f}\right) e^{-\|f\|^2\slash 2}df\\
&&\qquad =\mathbb E \left[\left(\sqrt{\operatorname{tr}(\mathcal U(x))}+\sqrt{X^TO^T\mathcal U(x)O X}\right)\left(\sqrt{\operatorname{tr}(\mathcal U(y))}+\sqrt{X^TO^T\mathcal U(y)O X}\right)\right]
\end{eqnarray*}
for $X=(0,0,X_3,\dots ,X_n)^T$, where $X_3,\dots ,X_n$ are standard normal independent random variables. By the Cauchy-Schwarz inequality we obtain
\[\mathbb E \left[\sqrt{X^TO^T\mathcal U(x)O X}\sqrt{X^TO^T\mathcal U(y)O X}\right]\leq \sqrt{\mathbb E \left[X^TO^T\mathcal U(x)O X\right]\mathbb E \left[X^TO^T\mathcal U(y)O X\right]}= \sqrt{\operatorname{tr}(\mathcal U'(x))\operatorname{tr}(\mathcal U'(y))},\]
where $\mathcal U'(x)=O^T\mathcal U(x)O\Sigma$ with $\Sigma = \operatorname{diag}(0,0,1,\dots ,1)$. Recalling that by properties of the trace and due to the fact that $$O\mathcal U(x)O^T=(O\partial_1U(x)) (O\partial_1U(x))^T+(O\partial_2U(x)) (O\partial_2U(x))^T$$ has nonnegative diagonal entries, we have
\[\operatorname{tr}(\mathcal U'(x))=\operatorname{tr}(O^T\mathcal U(x)O)-(O^T\mathcal U(x)O)_{11}-(O^T\mathcal U(x)O)_{22}\leq \operatorname{tr}(O^T\mathcal U(x)O)=\operatorname{tr}(\mathcal U(x)),\]
and thus,
\[\operatorname{tr}(\mathcal U'(x))\lesssim \operatorname{tr}(\mathcal U(x)) =\sum_{i=1}^n (\partial_1U(x))_i^2+\sum_{i=1}^n (\partial_2U(x))_i^2=\|\partial_1 U(x)\|^2+\|\partial_2 U(x)\|^2=2A_{2^j}. \]
Therefore, we arrive at the bound
\[\mathbb E \left[\sqrt{X^TO^T\mathcal U(x)O X}\sqrt{X^TO^T\mathcal U(y)O X}\right]\lesssim A_{2^j}.\]
Similarly, we can see that
\[\mathbb E \left[\sqrt{X^TO^T\mathcal U(x) OX}\right]\leq \sqrt{\mathbb E \left[X^TO^T\mathcal U(x) OX\right] }\lesssim \sqrt{A_{2^j}}.\]
In total, we can bound the inner integral by $A_{2^j}O(1)$. We obtain thus
\[\frac{1}{(2\pi)^{n\slash 2}}\int_{\substack{|\tilde\beta_j(x)-z|<\varepsilon_1 \\ |\tilde\beta_j(y)-z|<\varepsilon_2 }} \|\nabla\tilde{\beta}_j(x)\|\|\nabla\tilde{\beta}_j(y)\|e^{-\|\tilde\beta_j\|^2\slash 2}d\tilde\beta_j\lesssim A_{2^j}|P_z|\lesssim A_{2^j}\varepsilon_1\varepsilon_2 \frac{1}{\sqrt{1-u(x,y)^2}} .\]
Consequently, we have
\[K_{\varepsilon_1\varepsilon_2}(x,y) \lesssim A_{2^j} \frac{1}{\sqrt{1-u(x,y)^2}}\]
and
\[K(x,N)\lesssim  A_{2^j} \frac{1}{\sqrt{1-u(x,N)^2}}.\]
Note that with our above definitions we have
\[u(x,N)=\frac{1}{B_{2^j}}\sum_{\ell=2^{j-1}}^{2^{j+1}}b\left(\frac{\ell}{2^j}\right)^2C_\ell\frac{2\ell+1}{4\pi}P_\ell(\cos\theta).\]
Therefore, the integral for small $\theta$ is bounded by
\[\int_0^{C/\ell}A_{2^j}\frac{1}{\sqrt{1-\left(\frac{1}{B_{2^j}}\sum_{\ell=2^{j-1}}^{2^{j+1}}b\left(\frac{\ell}{2^j}\right)^2C_\ell\frac{2\ell+1}{4\pi}P_\ell(\cos\theta)\right)^2}}\sin \theta d\theta.\]

We have
$$1-u^2= 1-\left(\frac{1}{B_{2^j}}\sum_{\ell=2^{j-1}}^{2^{j+1}}b\left(\frac{\ell}{2^j}\right)^2C_\ell\frac{2\ell+1}{4\pi}P_\ell(\cos\theta)\right)^2=$$$$ \left(\frac{1}{B_{2^j}^2}\sum_{\ell=2^{j-1}}^{2^{j+1}}\sum_{\ell'=2^{j-1}}^{2^{j+1}}b\left(\frac{\ell'}{2^j}\right)^2C_{\ell'}\frac{2\ell'+1}{4\pi}b\left(\frac{\ell}{2^j}\right)^2C_\ell\frac{2\ell+1}{4\pi}(1-P_\ell(\cos\theta)P_{\ell'}(\cos\theta))\right).$$
From the Taylor approximation $P_\ell(\cos\theta)=1-\theta^2\frac{\ell(\ell+1)}{2}+\theta^4O(\ell^4)$ we conclude
\begin{eqnarray*}
&&1-P_\ell(\cos\theta)P_{\ell'}(\cos\theta)\\&&=1-\bigg(1-\theta^2 \frac{\ell(\ell+1)}{2}+O(\ell^4 \theta^4) -\theta^2 \frac{\ell'(\ell'+1)}{2}+O(\ell'^4 \theta^4)+ \theta^4 \frac{\ell'(\ell'+1)}{2}\frac{\ell(\ell+1)}{2}\bigg)\\&&=\theta^2 \frac{\ell(\ell+1)}{2} +\theta^2 \frac{\ell'(\ell'+1)}{2}+ \theta^4 \frac{\ell'(\ell'+1)}{2}\frac{\ell(\ell+1)}{2}+O(\ell^4 \theta^4)+O(\ell'^4 \theta^4),
\end{eqnarray*}
and thus,
\begin{eqnarray*}
1-u^2&=&\left(\frac{1}{B_{2^j}^2}\sum_{\ell=2^{j-1}}^{2^{j+1}}\sum_{\ell'=2^{j-1}}^{2^{j+1}}b\left(\frac{\ell'}{2^j}\right)^2C_{\ell'}\frac{2\ell'+1}{4\pi}b\left(\frac{\ell}{2^j}\right)^2C_\ell\frac{2\ell+1}{4\pi}\right) \bigg(\theta^2 \frac{\ell(\ell+1)}{2} \\&&+\theta^2 \frac{\ell'(\ell'+1)}{2}+O(\ell^4 \theta^4)+O(\ell'^4 \theta^4)+ \theta^4  \frac{\ell'(\ell'+1)}{2}\frac{\ell(\ell+1)}{2}\bigg)\\&=& \frac{1}{B_{2^j}^2} \left( 2\theta^2 A_{2^j}B_{2^j}^2+ O(\theta^4A_{2^j}^2 B_{2^j}^2) \right)=  2\theta^2 A_{2^j}\left( 1+ O(\theta^2A_{2^j}) \right).
\end{eqnarray*}
It follows that

\[\int_0^{C/\ell}  A_{2^j}\frac{1}{\sqrt{1-u(x,N)^2}}\sin\theta d\theta=\frac{1}{(2\pi)^{n\slash 2}}\int_0^{C/\ell}  A_{2^j}\frac{1}{\theta  \sqrt{2A_{2^j}}\sqrt{1+O(A_{2^j} \theta^2)}}\sin\theta d\theta \]\[= \sqrt{A_{2^j}} \int_0^{C/\ell} O(1) \frac{\sin \theta}{\theta} d\theta \lesssim  \frac{\sqrt{A_{2^j}}}{\ell}=O(1).\]
\end{proof}

\begin{proof}[Proof of Proposition \ref{chaosesq}]
 We have
\begin{eqnarray*}
&&\mathbb E[\operatorname{proj}(\mathcal L_j(z)|C_q)^2]=A_{2^j}\sum_{u_1=0}^q\sum_{k_1=0}^{u_1}\sum_{u_2=0}^q\sum_{k_2=0}^{u_2}\frac{\alpha_{k_1,\,u_1-k_1}\beta_{q-u_1}(z)}{k_1!(u_1-k_1)!(q-u_1)!}\frac{\alpha_{k_2,\,u_2-k_2}\beta_{q-u_2}(z)}{k_2!(u_2-k_2)!(q_2-u_2)!}\\
&&\times \int_{\mathbb S^2}\int_{\mathbb S^2}\mathbb E [H_{q-u_1}(\tilde{\beta}_j(x))H_{k_1}(\tilde\partial_1 \tilde{\beta}_j(x))H_{u_1-k_1}(\tilde\partial_2 \tilde{\beta}_j(x))H_{q-u_2}(\tilde{\beta}_j(y))H_{k_2}(\tilde\partial_1 \tilde{\beta}_j(y))H_{u_2-k_2}(\tilde\partial_2 \tilde{\beta}_j(y))]dxdy\\
&&=:A_{2^j}\sum_{u_1=0}^q\sum_{k_1=0}^{u_1}\sum_{u_2=0}^q\sum_{k_2=0}^{u_2}\frac{\alpha_{k_1,\,u_1-k_1}\beta_{q-u_1}(z)}{k_1!(u_1-k_1)!(q-u_1)!}\frac{\alpha_{k_2,\,u_2-k_2}\beta_{q-u_2}(z)}{k_2!(u_2-k_2)!(q_2-u_2)!}I_{q,\,j}^{(u_1,\,k_1,\,u_2,\,k_2)}.
\end{eqnarray*}
We need to understand the asymptotics of the terms $A_{2^j}I_{q,\,j}^{(u_1,\,k_1,\,u_2,\,k_2)}$. First note that due to isotropy of $f$ we can reparametrise the integrals as
\begin{eqnarray*}
&& I_{q,\,j}^{(u_1,\,k_1,\,u_2,\,k_2)}=8\pi^2\\
&&\times \int_0^\pi \mathbb E [H_{q-u_1}(\tilde{\beta}_j(x))H_{k_1}(\tilde\partial_1 \tilde{\beta}_j(x))H_{u_1-k_1}(\tilde\partial_2 \tilde{\beta}_j(x))H_{q-u_2}(\tilde{\beta}_j(y))H_{k_2}(\tilde\partial_1 \tilde{\beta}_j(y))H_{u_2-k_2}(\tilde\partial_2 \tilde{\beta}_j(y))] \sin \theta d\theta,
\end{eqnarray*}
where $x=(0,\,0)$ is the north pole of the sphere $\mathbb S^2$ and $y=(\theta ,\, 0)$.
By Lemma \ref{constdiag} (and with its notation) this integral (divided by $8\pi^2$) becomes
\begin{eqnarray*}
\sum_{\alpha} M_{\alpha}\int_0^\pi \mathbb E [\tilde{\beta}_j(x)\tilde{\beta}_j(y)]^\alpha\mathbb E [\tilde{\beta}_j(x)\tilde\partial_1 \tilde{\beta}_j(y) ]^\beta \mathbb E [\tilde\partial_1 \tilde{\beta}_j(x)\tilde\partial_1 \tilde{\beta}_j(y)]^\gamma \mathbb E [\tilde\partial_2 \tilde{\beta}_j(x)\tilde\partial_2 \tilde{\beta}_j(y)]^{\delta}\sin \theta d\theta .
\end{eqnarray*}

By \eqref{covbeta} we can write
\begin{eqnarray*}
&&\int_0^\pi \mathbb E [\tilde{\beta}_j(x)\tilde{\beta}_j(y)]^\alpha\mathbb E [\tilde{\beta}_j(x)\tilde\partial_1 \tilde{\beta}_j(y) ]^\beta \mathbb E [\tilde\partial_1 \tilde{\beta}_j(x)\tilde\partial_1 \tilde{\beta}_j(y)]^\gamma \mathbb E [\tilde\partial_2 \tilde{\beta}_j(x)\tilde\partial_2 \tilde{\beta}_j(y)]^{\delta}\sin \theta d\theta \\
&&=\left(\frac{1}{B_{2^j}}\right)^q\left(\frac{1}{A_{2^j}}\right)^{\frac{\beta}{2}+\gamma+\delta}\sum_{\gamma_1=0}^\gamma \binom{\gamma_1}{\gamma}(-1)^{\beta+(\gamma-\gamma_1)}\sum_{\ell_1,\dots , \ell_q=2^{j-1}}^{2^{j+1}}\prod_{i=1}^{q}b\left(\frac{\ell_i}{2^j}\right)^2 C_{\ell_i}\frac{2\ell_i+1}{4\pi}\\
&&\quad\times \int_0^\pi P_{\ell_1}(\cos\theta)\dots P_{\ell_\alpha}(\cos \theta)(\sin \theta)^{\beta}P'_{\ell_{\alpha+1}}(\cos\theta)\dots P'_{\ell_{\alpha+\beta}}(\cos\theta)\\
&&\quad\times P'_{\ell_{\alpha+\beta+1}}(\cos\theta)\dots P'_{\ell_{\alpha+\beta+\gamma_1}}(\cos\theta)(\cos \theta)^{\gamma_1} P''_{\ell_{\alpha+\beta+\gamma_1+1}}(\cos\theta)\dots P''_{\ell_{\alpha+\beta+\gamma}}(\cos\theta)(\sin\theta)^{2(\gamma-\gamma_1)}\\
&&\quad\times P'_{\ell_{\alpha+\beta+\gamma}}(\cos\theta)\dots P'_{\ell_{q}}(\cos\theta)\sin\theta d\theta\\
&&=:\left(\frac{1}{B_{2^j}}\right)^q\left(\frac{1}{A_{2^j}}\right)^{\frac{\beta}{2}+\gamma+\delta}\sum_{\gamma_1=0}^\gamma \binom{\gamma_1}{\gamma}(-1)^{\beta+(\gamma-\gamma_1)}\sum_{\ell_1,\dots , \ell_q=2^{j-1}}^{2^{j+1}}\prod_{i=1}^{q}b\left(\frac{\ell_i}{2^j}\right)^2 C_{\ell_i}\frac{2\ell_i+1}{4\pi}\operatorname{Int}_{\alpha\beta\gamma}^{(\gamma_1)}.
\end{eqnarray*}
 Note that the functions $P_n(\cos(\cdotp))$, $P'_n(\cos(\cdotp))$, $P''_n(\cos(\cdotp))$, $\cos(\cdotp)$ and $\sin(\cdotp)$ are either line- or point-symmetric at $\frac{\pi}{2}$, and therefore the integrals for each of the constellations will be either zero or twice the integrals between zero and $\frac{\pi}{2}$. One can see from \eqref{covbeta} and the identity
$$P'_n(\cos \theta)\cos\theta - P''_n(\cos\theta)\sin^2\theta =(n+1)^2P_n(\cos\theta)-P'_{n+1}(\cos\theta) $$
that the integrals cancel if and only if $\sum_{i=1}^q\ell_i$ is odd. Therefore, for the asymptotics it suffices to consider integrals ranging from zero to $\frac{\pi}{2}$. Moreover, let us substitute $\theta$ by $\frac{\psi}{\ell}$, where $\ell=2^j$. We obtain for $(\ell_1\dots , \ell_q)$ such that the integrand is line-symmetric
\begin{eqnarray*}
	\operatorname{Int}_{\alpha\beta\gamma}^{(\gamma_1)}&=&\frac{2}{\ell}\int_0^{\ell\pi\slash 2}P_{\ell_1}(\cos\frac{\psi}{\ell})\dots P_{\ell_\alpha}(\cos \frac{\psi}{\ell})(\sin \frac{\psi}{\ell})^{\beta}P'_{\ell_{\alpha+1}}(\cos\frac{\psi}{\ell})\dots P'_{\ell_{\alpha+\beta}}(\cos\frac{\psi}{\ell})\\
	&&\quad\times P'_{\ell_{\alpha+\beta+1}}(\cos\frac{\psi}{\ell})\dots P'_{\ell_{\alpha+\beta+\gamma_1}}(\cos\frac{\psi}{\ell})(\cos \frac{\psi}{\ell})^{\gamma_1} P''_{\ell_{\alpha+\beta+\gamma_1+1}}(\cos\frac{\psi}{\ell})\dots P''_{\ell_{\alpha+\beta+\gamma}}(\cos\frac{\psi}{\ell})\\
	&&\quad\times (\sin\frac{\psi}{\ell})^{2(\gamma-\gamma_1)} P'_{\ell_{\alpha+\beta+\gamma}}(\cos\frac{\psi}{\ell})\dots P'_{\ell_{q}}(\cos\frac{\psi}{\ell})\sin\frac{\psi}{\ell} d\psi=:\operatorname{IntS}_{\alpha\beta\gamma}^{(\gamma_1)}+\operatorname{IntL}_{\alpha\beta\gamma}^{(\gamma_1)},
\end{eqnarray*}
where the two terms $\operatorname{IntS}_{\alpha\beta\gamma}^{(\gamma_1)}$and $\operatorname{IntL}_{\alpha\beta\gamma}^{(\gamma_1)}$ represent the integral ranging from zero to $\varepsilon$ and from $\varepsilon$ to $\ell\pi\slash 2$ respectively. For the integrals away from zero we can use the asymptotics given in Proposition \ref{P-asympt} and expand the sum in the expansion of the second derivatives. Denoting by $\sim$ asymptotic equivalence for large $\ell$, we can write
\begin{eqnarray*}
	\operatorname{IntL}_{\alpha\beta\gamma}^{(\gamma_1)}&\sim& \sum_{\gamma'_2=0}^{\gamma_2}\binom{\gamma'_2}{ \gamma_2}(-1)^{\gamma'_1}\frac{2}{\ell}\int_\varepsilon^{\ell\pi\slash 2}\sin \left(\psi + \frac{\pi}{4}\right)^\alpha \left(\frac{2}{\pi}\right)^{\frac{\alpha}{2}}\left(\frac{1}{\sin(\psi \slash l)}\right)^{\frac{\alpha}{2}}\prod_{i=1}^\alpha \sqrt{\frac{1}{\ell_i}}\\
	&&\times\sin\left(\frac{\psi}{\ell}\right)^\beta \left(\frac{2}{\pi}\right)^{\frac{\beta}{2}}\sin \left(\psi - \frac{\pi}{4}\right)^\beta \left(\frac{1}{\sin(\psi \slash \ell)}\right)^{\frac{3\beta}{2}}\prod_{i=\alpha+1}^{\alpha+\beta} \sqrt{\ell_i}\\
	&&\times\cos\left(\frac{\psi}{\ell}\right)^{\gamma_1} \left(\frac{2}{\pi}\right)^{\frac{\gamma_1}{2}}\sin \left(\psi - \frac{\pi}{4}\right)^{\gamma_1} \left(\frac{1}{\sin(\psi \slash \ell)}\right)^{\frac{3\gamma_1}{2}}\prod_{i=\alpha+\beta + 1}^{\alpha+\beta+\gamma_1} \sqrt{\ell_i}\\
	&&\times \prod_{i=\alpha+\beta+\gamma_1+1}^{\alpha+\beta+\gamma_1+\gamma'_2}\left(\ell_i^2 \sqrt{\frac{1}{\ell_i}}\right) \left(\frac{1}{\sin(\psi \slash \ell)}\right)^{2\gamma'_2} \sin \left(\psi + \frac{\pi}{4}\right)^{\gamma'_2} \left(\frac{2}{\pi}\right)^{\frac{\gamma'_2}{2}}\left(\frac{1}{\sin(\psi \slash \ell)}\right)^{\frac{\gamma'_2}{2}}\\
	&&\times \left(\frac{2}{\sin(\psi \slash \ell)}\right)^{2\gamma''_2} \left(\frac{2}{\pi}\right)^{\frac{\gamma''_2}{2}}\sin \left(\psi - \frac{\pi}{4}\right)^{\gamma''_2} \left(\frac{1}{\sin(\psi \slash \ell)}\right)^{\frac{3\gamma''_2}{2}}\prod_{i=\alpha+\beta+\gamma_1+\gamma'_2+1}^{\alpha+\beta+\gamma} \sqrt{\ell_i}\\
	&&\times \sin\left(\frac{\psi}{\ell}\right)^{2\gamma_2} \left(\frac{2}{\pi}\right)^{\frac{\delta}{2}}\sin \left(\psi - \frac{\pi}{4}\right)^\delta \left(\frac{1}{\sin(\psi \slash \ell)}\right)^{\frac{3\delta}{2}}\prod_{i=\alpha+\beta+\gamma+1}^{q} \sqrt{\ell_i}  \sin\left(\frac{\psi}{\ell}\right)d\psi\\
	&&=:\sum_{\gamma'_2=0}^{\gamma_2} \operatorname{IntL}_{\alpha\beta\gamma}^{(\gamma_1,\,\gamma'_2)}
\end{eqnarray*}
with $\gamma_2=\gamma-\gamma_1$ and $\gamma''_2=\gamma_2-\gamma'_2$. Now use the asymptotics $\sin (x) \approx x$ as well as $\cos (x)\approx 1$ and write
\begin{eqnarray*}
\operatorname{IntL}_{\alpha\beta\gamma}^{(\gamma_1,\,\gamma'_2)}&\sim &\frac{1}{\ell}\int_\varepsilon^{\ell\pi\slash 2} C_{\psi} \left(\frac{\ell}{\psi}\right)^{\frac{\alpha}{2}}\prod_{i=1}^\alpha \sqrt{\frac{1}{\ell_i}} \left(\frac{\psi}{\ell}\right)^{\beta}\prod_{i=\alpha+1}^{\alpha+\beta} \sqrt{\ell_i}\left(\frac{\ell}{\psi}\right)^{\frac{3\beta}{2}}\\
&&\times \prod_{i=\alpha+\beta + 1}^{\alpha+\beta+\gamma_1} \sqrt{\ell_i} \left(\frac{\ell}{\psi}\right)^{\frac{3\gamma_1}{2}}\prod_{i=\alpha+\beta+\gamma_1+1}^{\alpha+\beta+\gamma_1+\gamma'_2}\left(\ell_i^2 \sqrt{\frac{1}{\ell_i}}\right) \left(\frac{\ell}{\psi}\right)^{2\gamma'_2}\left(\frac{\ell}{\psi}\right)^{\frac{\gamma'_2}{2}}\\
&&\times \left(\frac{\ell}{\psi}\right)^{2\gamma''_2}\prod_{i=\alpha+\beta+\gamma_1+\gamma'_2+1}^{\alpha+\beta+\gamma} \sqrt{\ell_i} \left(\frac{\ell}{\psi}\right)^{\frac{3\gamma''_2}{2}}\left(\frac{\psi}{\ell}\right)^{2\gamma_2}\prod_{i=\alpha+\beta+\gamma+1}^{q} \sqrt{\ell_i} \left(\frac{\ell}{\psi}\right)^{\frac{3\delta}{2}}\frac{\psi}{\ell} d\psi,
\end{eqnarray*}
where $\sim$ means asymptotic equivalence up to a constant with respect to $\ell$ and
\begin{eqnarray*}
	C_\psi = \sin \left(\psi + \frac{\pi}{4}\right)^{\alpha+\gamma'_2} \sin \left(\psi - \frac{\pi}{4}\right)^{\beta+\gamma_1+\gamma''_2+\delta}.
\end{eqnarray*}
For $q>4$ and $q=3$ the integral
\begin{eqnarray*}
	\int_\varepsilon^{\ell\pi\slash 2} C_{\psi}\psi ^{\alpha + \beta+\gamma'_2-\frac{3}{2}q+1}d\psi ,
\end{eqnarray*}
 in which all the factors of $\operatorname{IntL}_{\alpha\beta\gamma}^{(\gamma_1,\,\gamma'_2)}$ depending on $\psi$ are collected, is finite (see Lemma \ref{finiteint} for the proof). Therefore, we have
 \begin{eqnarray*}
 \operatorname{IntL}_{\alpha\beta\gamma}^{(\gamma_1,\,\gamma'_2)}&\sim &l^{\frac{3}{2}q-\alpha-\beta-\gamma'_2-2}\prod_{i=1}^\alpha \sqrt{\frac{1}{\ell_i}}\prod_{i=\alpha + 1}^{\alpha+\beta+\gamma_1} \sqrt{\ell_i}\prod_{i=\alpha+\beta+\gamma_1+1}^{\alpha+\beta+\gamma_1+\gamma'_2}\ell_i^{\frac{3}{2}}\prod_{i=\alpha+\beta+\gamma_1+\gamma'_2+1}^{q} \sqrt{\ell_i}.
 \end{eqnarray*}
Now recall that using the asymptotics given in Remark \ref{sum-asympt} one can write
\begin{eqnarray*}
	\sum_{\ell_1=2^{j-1}}^{2^{j+1}}b\left(\frac{\ell_1}{2^j}\right)^2 C_{\ell_1}\frac{2\ell_1+1}{4\pi}\sqrt{\frac{1}{\ell_1}}\sim \ell^{2-\frac{1}{2}-a}.
\end{eqnarray*}
Performing a similar calculation for other $\ell_i$ ($i=2,\dots ,q$), we arrive at
\begin{eqnarray*}
&&\sum_{\substack{\ell_1,\dots , \ell_q=2^{j-1}\\\sum \ell_i \text{ even}}}^{2^{j+1}}\prod_{i=1}^{q}b\left(\frac{\ell_i}{2^j}\right)^2 C_{\ell_i}\frac{2\ell_i+1}{4\pi}\operatorname{Int}_{\alpha\beta\gamma}^{(\gamma_1,\,\gamma'_2)}\\
&&\quad\sim \ell^{\frac{3}{2}q-\alpha-\beta-\gamma'_2-2+q(2-a)-\frac{1}{2}\alpha+\frac{1}{2}(\beta+\gamma_1+\gamma''_2+\delta)+\frac{3}{2}\gamma'_2},
\end{eqnarray*}
and consequently, using the asymptotic results for $A_{2^j}$ and $B^{2^j}$,
\begin{eqnarray*}
	A_{2^j}\left(\frac{1}{B_{2^j}}\right)^q\left(\frac{1}{A_{2^j}}\right)^{\frac{\beta}{2}+\gamma+\delta}\sum_{\ell_1,\dots , \ell_q=2^{j-1}}^{2^{j+1}}\prod_{i=1}^{q}b\left(\frac{\ell_i}{2^j}\right)^2 C_{\ell_i}\frac{2\ell_i+1}{4\pi}\operatorname{Int}_{\alpha\beta\gamma}^{(\gamma_1,\,\gamma'_2)}\sim \ell^0.
\end{eqnarray*}
For $\operatorname{IntS}$ we note that the asymptotic formula from \cite{Sz} for $\theta\in (0,\varepsilon\slash \ell)$
\[P_n(\cos \theta)=\left(\frac{\theta}{\sin\theta}\right)^{1\slash 2}J_0((n+1\slash 2)\theta)+\delta_n(\theta)\]
with $\delta_n(\theta)\lesssim \theta^2\leq \frac{\varepsilon^2}{\ell^2}$ and a Bessel function $J_0$ can be transferred via classical identities (see \cite{AS})
\[P'_n(x)=\frac{n}{1-x^2}(P_{n-1}(x)-xP_{n}(x))\]
and
\[P''_n(x)=\frac{1}{1-x^2}(2xP'_n(x)-n(n+1)P_n(x))\]
to
\[P'_n(\cos\theta)=\frac{n}{\sin^2\theta}\left(\left(\frac{\theta}{\sin\theta}\right)^{1\slash 2}J_0((n-1\slash 2)\theta)+\delta_{n-1}(\theta)-\cos\theta \left(\frac{\theta}{\sin\theta}\right)^{1\slash 2}J_0((n+1\slash 2)\theta)-\cos \theta\delta_n(\theta)\right)\]
and
\begin{eqnarray*}
&& P'_n(\cos\theta)\cos\theta -P''_n(\cos\theta)\sin^2\theta =  n(n+1)P_n(\cos \theta)-P'_n(\cos\theta)\cos\theta\\
&&\qquad = n(n+1)\left(\frac{\theta}{\sin\theta}\right)^{1\slash 2}J_0((n+1\slash 2)\theta)+n(n+1)\delta_n(\theta)\\
&&\qquad\qquad -\frac{n\cos\theta}{\sin^2\theta}\left(\left(\frac{\theta}{\sin\theta}\right)^{1\slash 2}J_0((n-1\slash 2)\theta)+\delta_{n-1}(\theta)-\cos\theta \left(\frac{\theta}{\sin\theta}\right)^{1\slash 2}J_0((n+1\slash 2)\theta)-\cos \theta\delta_n(\theta)\right).
\end{eqnarray*}
After the substitution $\theta=\frac{\psi}{\ell}$ the terms that one then needs to consider in order to establish convergence are of the form
\[\sum_{\ell=2^{j-1}}^{2^{j+1}}b\left(\frac{\ell}{2^j}\right)^2C_\ell \frac{2\ell+1}{4\pi}\ell^p J_0\left((\ell \pm 1\slash 2)\frac{\psi}{2^j}\right),\]
and for these terms we have similarly to Remark \ref{sum-asympt}
\begin{eqnarray}\label{limSmall}
  &&\lim_{j\to\infty}(2^j)^{a-2-p}\sum_{\ell=2^{j-1}}^{2^{j+1}}b\left(\frac{\ell}{2^j}\right)^2C_\ell \frac{2\ell+1}{4\pi}\ell^p J_0\left((\ell \pm 1\slash 2)\frac{\psi}{2^j}\right)\nonumber \\
  &&\qquad = \frac{G}{2\pi}\int_{1\slash 2}^2 b^2(x)x^{p+1-a}J_0(x\psi)dx.
\end{eqnarray}
Note that the contribution of terms containing $\delta_n$ is smaller than that of the others and can be ignored. Moreover, we have
	\begin{eqnarray*}
		&&A_{2^j}\int_0^{\frac{\varepsilon}{\ell}} \mathbb E [\tilde{\beta}_j(x)\tilde{\beta}_j(y)]^\alpha\mathbb E [\tilde{\beta}_j(x)\tilde\partial_1 \tilde{\beta}_j(y) ]^\beta \mathbb E [\tilde\partial_1 \tilde{\beta}_j(x)\tilde\partial_1 \tilde{\beta}_j(y)]^\gamma \mathbb E [\tilde\partial_2 \tilde{\beta}_j(x)\tilde\partial_2 \tilde{\beta}_j(y)]^{\delta}\sin \theta d\theta \\ &&\leq A_{2^j}\int_0^{\frac{\varepsilon}{\ell}} |\sin \theta| \,d\theta = \dfrac{A_{2^j}}{\ell} \int_0^{\varepsilon} \bigg|\sin \frac{\psi}{\ell} \bigg| \,d\psi \sim \dfrac{A_{2^j}}{\ell^2}
	\end{eqnarray*}
by normalisation, i.e. the integrals close to zero are bounded. Therefore, the order with respect to $2^j$ resulting from the evaluation of terms in \eqref{limSmall} is at most constant and the limit with respect to $j$ is well-defined. This finishes the proof.
\end{proof}

\begin{lemma}\label{finiteint}
With the notation from Proposition \ref{chaosesq} set $\zeta:=\alpha+ \gamma_2'$. Then for $q=3$, $q >4$, and \begin{equation}\begin{cases}q=4 \\ \zeta=1,3 \end{cases}\end{equation} we have that
\begin{eqnarray}\label{constant}
	\int_\varepsilon^{\ell \pi\slash 2} C_{\psi}\psi ^{\alpha + \beta+\gamma'_2-\frac{3}{2}q+1}d\psi <\infty,
\end{eqnarray}
where $$C_{\psi}=  \sin \left(\psi + \frac{\pi}{4}\right)^{\alpha+\gamma'_2} \sin \left(\psi - \frac{\pi}{4}\right)^{\beta+\gamma_1+\gamma''_2+\delta}.$$
\end{lemma}
\begin{proof}
    It is easily seen that the constant $C_\psi$ can be bounded by 1 and then if $\alpha + \beta+\gamma'_2-\dfrac{3}{2}q+1 <-1$ the integral in (\ref{constant}) converges. Writing $q=\alpha+\beta+\gamma_1+\gamma_2'+\gamma_2''+\delta$ this condition is satisfied if $\dfrac{q}{2} > 2-\gamma_1-\gamma_2''-\delta,$ which holds for $q>4$ or $q=3,4$ and at least one among $\gamma_1,\gamma_2'',\delta$ different from zero.\\

Let us consider the case $\gamma_1=\gamma_2''=\delta=0$. The integral we need to compute is
\begin{equation}\label{int2}
\int_\varepsilon^{\ell \pi\slash 2}  \sin \left(\psi + \frac{\pi}{4}\right)^{\alpha+\gamma'_2} \sin \left(\psi - \frac{\pi}{4}\right)^{\beta} \psi ^{\alpha + \beta+\gamma'_2-\frac{3}{2}q+1}d\psi.
\end{equation} 
Under this condition we have that $q= \alpha+\beta+\gamma_2^\prime = \beta+\zeta$ which implies $\beta= q-\zeta$ and since $\sin\left( \psi-\frac{\pi}{4}\right)=\cos\left(\psi+\frac{\pi}{4}\right)$, the integral in (\ref{int2}) becomes
\begin{equation}\label{int1}
\int_\varepsilon^{\ell \pi\slash 2}  \sin \left(\psi + \frac{\pi}{4}\right)^{\zeta} \cos  \left(\psi + \frac{\pi}{4}\right)^{q-\zeta} \psi ^{-\frac{q}{2}+1}d\psi.
\end{equation}
Now if $q=3$,
we have 4 cases: $\zeta=0,1,2,3$.
If $\zeta=3$ we get 
$$\int_\varepsilon^{\ell \pi\slash 2} \sin \left(\psi + \frac{\pi}{4}\right)^{3} \psi ^{-\frac{1}{2}}d\psi = \int_\varepsilon^{\ell\pi\slash 2}  \sin \left(\psi + \frac{\pi}{4}\right) \sin  \left(\psi + \frac{\pi}{4}\right)^{2} \psi ^{-\frac{1}{2}}d\psi $$
$$= \int_\varepsilon^{\ell \pi\slash 2}  \sin \left(\psi + \frac{\pi}{4}\right) \psi^{-1/2} d\psi - \int_\varepsilon^{\ell \pi\slash 2}  \sin \left(\psi + \frac{\pi}{4}\right) \cos  \left(\psi + \frac{\pi}{4}\right)^2\psi^{-1/2} d\psi  $$
which equals by integration by parts
$$=   \cos \left(\psi + \frac{\pi}{4}\right) \psi^{-1/2} \bigg|_\varepsilon^{\ell \pi\slash 2}- \int_\varepsilon^{\ell\pi\slash 2}  \cos \left(\psi + \frac{\pi}{4}\right) \psi^{-3/2} d\psi $$$$- \cos  \left(\psi + \frac{\pi}{4}\right)^3\psi^{-1/2} \bigg|_\varepsilon^{\ell\pi\slash 2} +  \int_\varepsilon^{\ell\pi\slash 2}  \cos \left(\psi + \frac{\pi}{4}\right)^3 \psi^{-3/2} d\psi < \infty$$  as $\ell \to \infty$.
For $\zeta= 0  $ we can see in the same way that $$\int_\varepsilon^{\ell \pi\slash 2} \cos \left(\psi + \frac{\pi}{4}\right)^{3} \psi ^{-\frac{1}{2}}d\psi = \int_\varepsilon^{\ell\pi\slash 2}  \cos \left(\psi + \frac{\pi}{4}\right) \cos  \left(\psi + \frac{\pi}{4}\right)^{2} \psi ^{-\frac{1}{2}}d\psi $$
$$= \int_\varepsilon^{\ell \pi\slash 2}  \cos \left(\psi + \frac{\pi}{4}\right) \psi^{-1/2} d\psi - \int_\varepsilon^{\ell \pi\slash 2}  \cos \left(\psi + \frac{\pi}{4}\right) \sin  \left(\psi + \frac{\pi}{4}\right)^2\psi^{-1/2} d\psi  $$
$$=   \sin \left(\psi + \frac{\pi}{4}\right) \psi^{-1/2} \bigg|_\varepsilon^{\ell \pi\slash 2}- \int_\varepsilon^{\ell\pi\slash 2}  \sin \left(\psi + \frac{\pi}{4}\right) \psi^{-3/2} d\psi $$$$- \sin  \left(\psi + \frac{\pi}{4}\right)^3\psi^{-1/2} \bigg|_\varepsilon^{\ell\pi\slash 2} +  \int_\varepsilon^{\ell\pi\slash 2}  \sin \left(\psi + \frac{\pi}{4}\right)^3 \psi^{-3/2} d\psi < \infty.$$ 
For $\zeta= 1,2$ we obtain, respectively, $$\int_\varepsilon^{\ell \pi\slash 2}  \sin \left(\psi + \frac{\pi}{4}\right)^{2} \cos  \left(\psi + \frac{\pi}{4}\right) \psi ^{-\frac{1}{2}}d\psi $$ and $$\int_\varepsilon^{\ell \pi\slash 2}  \sin \left(\psi + \frac{\pi}{4}\right) \cos  \left(\psi + \frac{\pi}{4}\right)^{2} \psi ^{-\frac{1}{2}}d\psi $$ which converge, as we have just shown. \\

We consider $q=4$ and $\zeta=1,3$. Then, integral (\ref{int1}) becomes, respectively,
$$\int_\varepsilon^{\ell \pi\slash 2}  \sin \left(\psi + \frac{\pi}{4}\right) \cos  \left(\psi + \frac{\pi}{4}\right)^3 \psi ^{-1}d\psi $$ and
$$\int_\varepsilon^{\ell \pi\slash 2}  \sin \left(\psi + \frac{\pi}{4}\right)^{3} \cos  \left(\psi + \frac{\pi}{4}\right) \psi ^{-1}d\psi $$
 which can be proved to be convergent using, as before, integration by parts.
\end{proof}


\subsection{On the proof of CLT}\label{OnCLT}
In this section we give the proofs of the technical lemmas we used to prove the Central Limit Theorem.

\begin{proof}[Proof of Lemma \ref{Sym}]
We recall that
	\begin{eqnarray*}
		\operatorname{proj}(\mathcal L_j(z)|C_q)&=&\sqrt{A_{2^j}}\sum_{u=0}^q\sum_{k=0}^u\frac{\alpha_{k,\,u-k}\beta_{q-u}(z)}{k!(u-k)!(q-u)!}\\
		&&\quad \times \int_{S^2}H_{q-u}(\tilde{\beta}_j(x))H_k(\tilde\partial_1 \tilde{\beta}_j(x))H_{u-k}(\tilde\partial_2 \tilde{\beta}_j(x)) dx.
	\end{eqnarray*}
	Keeping in mind Section \ref{Malliavin}, we
replace Hermite polynomials by multiple integrals, hence we write
	\begin{eqnarray*}
		\operatorname{proj}(\mathcal L_j(z)|C_q)=\sqrt{A_{2^j}}\sum_{u=0}^q\sum_{k=0}^u\frac{\alpha_{k,\,u-k}\beta_{q-u}(z)}{k!(u-k)!(q-u)!}I_q(g_{q,j,k,u})=:I_q(g_{q,j})
	\end{eqnarray*}
	with
	$$g_{q,j,k,u}(y_1,\dots y_q)=\int_{\mathbb{S}^2}\prod_{m=1}^{q-u}\tilde\Theta^{(0)}_j(\langle x,\,y_m\rangle)\prod_{m=q-u+1}^{q-u+k}\tilde\Theta^{(1)}_j(\langle x,\,y_m\rangle)\prod_{m=q-u+k+1}^{q}\tilde\Theta^{(2)}_j(\langle x,\,y_m\rangle)d\sigma (x)$$
	and
	$$g_{q,j}=\sqrt{A_{2^j}}\sum_{u=0}^q\sum_{k=0}^u\frac{\alpha_{k,\,u-k}\beta_{q-u}(z)}{k!(u-k)!(q-u)!}g_{q,j,k,u}$$
with the notation from Section \ref{Malliavin}.
	For the derivative of the projection we thus obtain from formula \eqref{derivM}
	$$D_z\operatorname{proj}(\mathcal L_j(z)|C_q)=qI_{q-1}(\tilde g_{q,j} (y_1,\dots ,y_{q-1},z)),$$
	where $\tilde g_{q,j}$ denotes the symmetrisation of the function $g_{q,j}$. By definition of $L^{-1}$, using the multiplication formula and the definition of contraction, we conclude that
	\begin{eqnarray*}
		&&\langle D\operatorname{proj}(\mathcal L_j(z)|C_q),-DL^{-1}\operatorname{proj}(\mathcal L_j(z)|C_{q'})\rangle_{\mathcal H}=\frac{1}{q'}\langle D\operatorname{proj}(\mathcal L_j(z)|C_q),D\operatorname{proj}(\mathcal L_j(z)|C_{q'})\rangle_{\mathcal H}\\
		&&\qquad = q\langle I_{q-1}(\tilde g_{q,j} (y_1,\dots ,y_{q-1},z)), I_{q'-1}(\tilde g_{q',j} (y_1,\dots ,y_{q'-1},z))\rangle_{\mathcal H}\\
		&&\qquad = q\int_{\mathbb S^2} I_{q-1}(\tilde g_{q,j} (y_1,\dots ,y_{q-1},z)) I_{q'-1}(\tilde g_{q',j} (y_1,\dots ,y_{q'-1},z))d\sigma (z)\\
		&&\qquad = q\sum_{r=1}^{q\wedge q'}(r-1)!\binom{q-1}{ r-1}\binom{q'-1}{ r-1}I_{q+q'-2r}(\tilde g_{q,j} (y_1,\dots ,y_{q})\tilde\otimes_r \tilde g_{q',j} (y_1,\dots ,y_{q'}))
	\end{eqnarray*}
in view of (\ref{product}).	Consequently,
	\begin{eqnarray*}
		&&\mathbb{E}[\langle D\operatorname{proj}(\mathcal L_j(z)|C_q),-DL^{-1}\operatorname{proj}(\mathcal L_j(z)|C_{q'})\rangle_{\mathcal H}^2]\\
		&&\qquad =q^2\sum_{r=1}^{q\wedge q'}(r-1)!^2\binom{q-1}{ r-1}^2 \binom{q'-1}{ r-1}^2 (q+q'-2r)!\|\tilde g_{q,j} (y_1,\dots ,y_{q})\tilde\otimes_r \tilde g_{q',j} (y_1,\dots ,y_{q'})\|^2_{\mathcal{H}^{q+q'-2r}}\\
		&&\qquad \leq q^2\sum_{r=1}^{q\wedge q'}(r-1)!^2\binom{q-1}{ r-1}^2 \binom{q'-1}{ r-1}^2 (q+q'-2r)!\|\tilde g_{q,j} (y_1,\dots ,y_{q})\otimes_r \tilde g_{q',j} (y_1,\dots ,y_{q'})\|^2_{\mathcal{H}^{q+q'-2r}}.
	\end{eqnarray*}
	Now let us focus on a single contraction term:
	\begin{eqnarray*}
	&&	\tilde g_{q,j} \otimes_r \tilde g_{q',j} (y_1,\dots ,y_{q+q'-2r})\\
		&&\qquad =A_{2^j}\sum_{u=0}^q\sum_{k=0}^u  \sum_{u'=0}^{q'}\sum_{k'=0}^{u'} C_{quk}C_{q'u'k'} \tilde g_{q,j,k,u}(y_1,\dots ,y_{q})\otimes_r \tilde g_{q',j,k',u'}(y_1,\dots ,y_{q'})\\
		&&\qquad =A_{2^j}\sum_{u=0}^q\sum_{k=0}^u  \sum_{u'=0}^{q'}\sum_{k'=0}^{u'}C_{quk}C_{q'u'k'} \\
		&&\qquad\times \int_{(\mathbb S^2)^r}\tilde g_{q,j,k,u}(y_1,\dots ,y_{q-r},t_1,\dots ,t_r) \tilde g_{q',j,k',u'}(y_{q-r+1},\dots ,y_{q+q'-2r},t_1,\dots , t_r)d\sigma (t_1)\dots d\sigma(t_r).
	\end{eqnarray*}
	Each integral in this expression is a double sum over permutations of $q$ and $q'$ divided by $q!q'!$ of double integrals $d\sigma(x_1)d\sigma(x_2)$ emerging from the definition of $g_{qjuk}$. The entries at the last $r$ positions integrate to covariances of the form $\mathbb E [f_1(x_1)f_2(x_2)]$ with $f_1$ and $f_2$ being either $\tilde{\beta}_j$, $\tilde\partial_1\tilde\beta_j$ or $\tilde\partial_2\tilde\beta_j$. We can see that
	\begin{eqnarray}\label{intF}
		&&\|\tilde g_{q,j} \otimes_r \tilde g_{q',j}\|^2=A_{2^j}^2 \sum_{u_1=0}^q\sum_{k_1=0}^{u_1}  \sum_{u_1'=0}^{q'}\sum_{k_1'=0}^{u_1'}\sum_{u_2=0}^q\sum_{k_2=0}^{u_2}  \sum_{u_2'=0}^{q'}\sum_{k_2'=0}^{u_2'}C_{qu_1k_1}C_{q'u_1'k_1'}C_{qu_2k_2}C_{q'u_2'k_2'}\nonumber\\
		&& \times\frac{1}{(q!q'!)^2}\sum_{\sigma\sigma'\sigma_1\sigma_1'}\int_{(\mathbb S^2)^4}F_1(\langle x_1,x_3\rangle)F_2(\langle x_2,x_4\rangle)F_3(\langle x_1,x_2\rangle)F_4(\langle x_3,x_4\rangle)d\sigma(x_1)\dots d\sigma(x_4),\nonumber\\
	\end{eqnarray}
	where $F_1$, $F_2$, $F_3$ and $F_4$ are products of respectively $q-r$, $q'-r$, $r$ and $r$ of the covariances involving $\tilde{\beta}_j$, $\tilde\partial_1\tilde\beta_j$ and $\tilde\partial_2\tilde\beta_j$. The specific terms depend on the choice of permutations $\sigma$, $\sigma'$, $\sigma_1$ and $\sigma_1'$.\\

Now we note that
	\begin{eqnarray*}
		&&\int_{\mathbb S^2}\left(\frac{K_M}{(1+2^jd(x,y))^M}\right)^pd\sigma (x)=K_M^p 2\pi \int_0^\pi \frac{1}{(1+2^j \theta )^{Mp}}\sin\theta d\theta\\
		&&\quad =K_M^p \left(\frac{1}{2^j}\right)^{Mp}2\pi\int_0^\pi \frac{1}{\left(\frac{1}{2^j}+\theta\right)^{Mp}}\sin\theta d\theta \leq 2\pi K_M^p C\left(\frac{1}{2^j}\right)^{Mp}
	\end{eqnarray*}
	for big $j$ and some constant $C$, since $\lim_{j\to\infty}\int_0^{\pi}\frac{1}{\left(\frac{1}{2^j}+\theta\right)^{Mp}}\sin\theta d\theta $ converges to a (nonzero) constant for every $M\in\mathbb N$. Thus, for every $p\neq 0$ and $M>1$ we obtain a convergent bound.\\
	
	Going back to \eqref{intF}, recall that each factor in $F_1$, $F_2$, $F_3$, $F_4$ is bounded by at most $\frac{K_M2^{2j}}{(1+2^jd(x,y))^M}$ for every $M\in\mathbb N$ with some constant $K_M$ as showed in Proposition \ref{boundCov}. Hence, we obtain the following bounds
	\begin{eqnarray*}
		\int_{\mathbb S^2} F_1(\langle x,y\rangle)d\sigma (x)&\leq & 2\pi C K_M^{q-r}(2^j)^{(2-M)(q-r)},\\
		\int_{\mathbb S^2} F_2(\langle x,y\rangle)d\sigma (x)&\leq & 2\pi C K_M^{q'-r}(2^j)^{(2-M)(q'-r)},\\
		\int_{\mathbb S^2} F_3(\langle x,y\rangle)d\sigma (x)&\leq & 2\pi C K_M^{r}(2^j)^{(2-M)r},\\
			\int_{\mathbb S^2} F_4(\langle x,y\rangle)d\sigma (x)&\leq & 2\pi C K_M^{r}(2^j)^{(2-M)r},
	\end{eqnarray*}
and similar estimates can be found for the squares of $F_i$.
	Moreover, since $x^a y^b \leq x^{a+b}+y^{a+b}$ for $x$, $y$ positive, we can bound
\begin{eqnarray}
\int_{(\mathbb S^2)^4}&&|F_1(\langle x_1,x_3\rangle)F_2(\langle x_2,x_4\rangle)F_3(\langle x_1,x_2\rangle)F_4(\langle x_3,x_4\rangle)|d\sigma(x_1)\dots d\sigma(x_4),\nonumber\\
&\leq &
\int_{(\mathbb S^2)^4}|F_1(\langle x_1,x_3\rangle)||F_2(\langle x_2,x_4\rangle)|F_3(\langle x_1,x_2\rangle)^2 d\sigma(x_1)\dots d\sigma(x_4),\nonumber\\
&+& \int_{(\mathbb S^2)^4}|F_1(\langle x_1,x_3\rangle)||F_2(\langle x_2,x_4\rangle)|F_4(\langle x_3,x_4\rangle)^2 d\sigma(x_1)\dots d\sigma(x_4),\nonumber\\
\end{eqnarray}
and then
	$$\int_{(\mathbb S^2)^4}|F_1(\langle x_1,x_3\rangle)F_2(\langle x_2,x_4\rangle)F_3(\langle x_1,x_2\rangle)F_4(\langle x_3,x_4\rangle)|d\sigma(x_1)\dots d\sigma(x_4)\leq (K'_M)^{q+q'} (2^j)^{(2-M)(q+q')}$$
	with some constant $K'_M$ depending only on $M$. We conclude that
	\begin{eqnarray*}
		&&\|\tilde g_{q,j} \otimes_r \tilde g_{q',j}\|^2\leq (K'_M)^{q+q'} A_{2^j}^2 (2^j)^{(2-M)(q+q')} \\
		&&\qquad \qquad\qquad\times\sum_{u_1=0}^q\sum_{k_1=0}^{u_1}  \sum_{u_1'=0}^{q'}\sum_{k_1'=0}^{u_1'}\sum_{u_2=0}^q\sum_{k_2=0}^{u_2}  \sum_{u_2'=0}^{q'}\sum_{k_2'=0}^{u_2'}C_{qu_1k_1}C_{q'u_1'k_1'}C_{qu_2k_2}C_{q'u_2'k_2'}\\
		&&\qquad \simeq \frac{(K'_M)^{q+q'}}{4}(2^j)^{(2-M)(q+q)+4}\sum_{u_1=0}^q\sum_{k_1=0}^{u_1}  \sum_{u_1'=0}^{q'}\sum_{k_1'=0}^{u_1'}\sum_{u_2=0}^q\sum_{k_2=0}^{u_2}  \sum_{u_2'=0}^{q'}\sum_{k_2'=0}^{u_2'}C_{qu_1k_1}C_{q'u_1'k_1'}C_{qu_2k_2}C_{q'u_2'k_2'}
	\end{eqnarray*}
	for big $j$, using the asymptotics of $A_{2^j}$. The second bound in \eqref{bound2} is obtained in a similar way from the inequality
	\begin{eqnarray*}
		&& \operatorname{Var}(\langle D\operatorname{proj}(\mathcal L_j(z)|C_q),-DL^{-1}\operatorname{proj}(\mathcal L_j(z)|C_{q})\rangle_{\mathcal H})\\
		&&\quad\leq q^2\sum_{r=1}^{q-1}(r-1)!^2\binom{q-1}{ r-1}^4 (2q-2r)!\|\tilde g_{q,j} (y_1,\dots ,y_{q})\otimes_r \tilde g_{q,j} (y_1,\dots ,y_{q})\|^2_{\mathcal{H}^{2q-2r}}.
	\end{eqnarray*}
\end{proof}

\begin{proof}[Proof of Lemma \ref{boundjq}]
	Let us consider the sum
	$$\sum_{k=0}^{u} \frac{\alpha_{k,u-k}}{k!(u-k)!}.$$
	By definition $\alpha_{k,u-k}$ is zero except for $k,u$ even. Hence we set $k=2k^\prime$ and $u=2u^\prime$ to write
	
	\begin{eqnarray}\label{1}
	\sum_{k=0}^{u} \frac{\alpha_{k,u-k}}{k!(u-k)!}&=& \sum_{k^\prime=0}^{u^\prime} \dfrac{\alpha_{2k^\prime 2u^\prime-2k^\prime}}{(2k^\prime)!(2u^\prime-2k^\prime)!}=\sum_{k^\prime=0}^{u^\prime} \sqrt{\frac{\pi}{2}} \dfrac{(2k^\prime)!(2u^\prime-2k^\prime)!}{(2k^\prime)!(2u^\prime-2k^\prime)! k^\prime! (\frac{2u^\prime-2k^\prime}{2})!} \frac{1}{2^{u^\prime}} P_{u^\prime}\left(\frac{1}{4}\right)\nonumber\\	&=&\frac{1}{2^{u^\prime}} P_{u^\prime}\left(\frac{1}{4}\right)\sqrt{\frac{\pi}{2}} 
	\sum_{k^\prime=0}^{u^\prime} \dfrac{1}{k^\prime! (\frac{2u^\prime-2k^\prime}{2})!}.
	\end{eqnarray}
	We note that
	\begin{equation*}\label{qq}
	\sum_{k^\prime=0}^{u^\prime} \dfrac{1}{k^\prime! (\frac{2u^\prime-2k^\prime}{2})!}= \sum_{k^\prime=0}^{u^\prime} \dfrac{1}{k^\prime! (u^\prime-k^\prime)!}=  \sum_{k^\prime=0}^{u^\prime} \dfrac{1}{k^\prime! (u^\prime-k^\prime)!} \frac{u^\prime!}{u^\prime!}
	= \frac{1}{u^\prime!} \sum_{k^\prime=0}^{u^\prime} \binom{u^\prime}{k^\prime} = \frac{2^{u^\prime}}{(u^\prime)!} \end{equation*}
	and then (\ref{1}) equals
	\begin{equation}\label{qqq}
 P_{u'}\left(\frac{1}{4}\right)\sqrt{\frac{\pi}{2}} \frac{1}{u'!}.
	\end{equation}
	Let us focus now on $P_{u'}\left(\frac{1}{4}\right)$; it can be checked that
	\begin{equation}\label{pu}
	P_{u'}\left(\frac{1}{4}\right):=\sum_{j=0}^{u'} (-1)^{j}(-1)^{u'} \binom{u'}{j} \frac{(2j+1)!}{j!^2} \left(\frac{1}{4}\right)^j=(-1)^{u'} \left( - \frac{(u'-3/2)!}{2\sqrt{\pi}u'!} \right).
	\end{equation}
Plugging \eqref{pu} into (\ref{qqq}) we get
	\begin{equation*}
	\sum_{k=0}^{u} \frac{\alpha_{k,u-k}}{k!(u-k)!}\leq  \sqrt{\frac{\pi}{2}} (-1)^{u'} \left( - \frac{(u'-3/2)!}{2\sqrt{\pi}u'!} \right)
	\frac{1}{u'!}.
	\end{equation*}
Since $|\frac{(u-3/2)!}{2\sqrt{\pi}u!}|\leq 1$  we obtain that
	\begin{equation}\label{bb}
\left|\sum_{u=0}^{q} \sum_{k=0}^{u} \frac{\alpha_{k,u-k}}{k!(u-k)!} \frac{\beta_{q-u}(z)}{(q-u)!}\right|
	\leq \sum_{u=0}^{[q/2]} \sqrt{\frac{\pi}{2}} 
	\frac{1}{u!}
	\frac{\beta_{q-2u}(z)}{(q-2u)!}.
	\end{equation}
	Now let us study the right hand side of \eqref{bb}. By definition of $\beta_{q}(z)$ we have 
	\begin{equation*}
	\sum_{u=0}^{[q/2]} \sqrt{\frac{\pi}{2}} 
	\frac{1}{u!}
	\frac{\beta_{q-2u}(z)}{(q-2u)!}= \sum_{u=0}^{[q/2]} \sqrt{\frac{\pi}{2}} 
	\frac{1}{u!}
	\frac{\phi(z) H_{q-2u}(z)}{(q-2u)!}.
	\end{equation*}
	Following the same lines as the argument in the proof of Theorem 9 in \cite{CM}, from the fact that for any finite $z$ we have $e^{-z^2/4} H_q(z) \leq \mbox{const } q^{q/2} e^{-q/2}$ as $q \rightarrow \infty$ (see eq. (4.14.9) \cite{lebedev}) and using Stirling's approximation to the factorial $(q-2u)!$, we can see that
	$$ \dfrac{\phi(z)}{(q-2u)!} \left[ \sqrt{\phi(z)} H_{q-2u}(z) \right]^2 \leq \dfrac{C \phi(z)}{\sqrt{q-2u}},$$ and therefore,
	$$\phi(z) H_{q-2u}(z)\leq \dfrac{C \sqrt{(q-2u)!}}{\sqrt[4]{(q-2u)}} \sqrt{\phi(z)}$$
so that
	\begin{equation}
	\sum_{u=0}^{[q/2]} \sqrt{\frac{\pi}{2}} 
	\frac{1}{u!}
	\frac{\phi(z) H_{q-2u}(z)}{(q-2u)!} \leq \sum_{u=0}^{[q/2]} \sqrt{\frac{\pi}{2}} 
	\frac{1}{u!}
	\frac{C \sqrt{\phi(z)}}{\sqrt{(q-2u)!} \sqrt[4]{q-2u}}.
	\end{equation}
	This last expression is bounded by
	$$C_z\sum_{u=0}^{[q/2]}\frac{1}{u!\sqrt{(q-2u)!}}$$
	with some constant $C_z$. Let us denote $q':=\lfloor q/2\rfloor$ and estimate this term by
	$$C'_z\sum_{u=0}^{q'}\frac{1}{u!\sqrt{(2q'-2u)!}}=C'_z\sum_{u=0}^{q'}\frac{1}{u!(q'-u)!}\sqrt\frac{(q'-u)!(q'-u)!}{(2q'-2u)!}=C'_z\frac{1}{q'!}\sum_{u=0}^{q'}\frac{\binom{q' }{ u}}{\sqrt{\binom{2q'-2u}{q'-u}}}.$$
	This last expression can be estimated by $C'_z\frac{1}{q'!}\sum_{u=0}^{q'}\binom{q'}{ u}$ which equals $C'_z\frac{1}{q'!}2^{q'}$. Note that for large $n$ we have $\frac{1}{n!}\sim 2^{n} n^{-1/4}\frac{1}{\sqrt{(2n)!}}$ by Stirling formula, and hence we obtain the bound
	$$\frac{1}{q'!}2^{q'}\leq 2^{q}\frac{1}{\sqrt{(q-1)!}}$$ which leads to \eqref{2}.
\end{proof}

\section*{Acknowledgments}
The authors would like to thank Domenico Marinucci for some useful discussions. R.S. has been supported by the German Research Foundation (DFG) via SFB 823 and RTG 2131. A.P.T. has been supported by the German Research Foundation (DFG) via RTG 2131 and by GNAMPA-INdAM (project: \emph{Stime asintotiche: principi di invarianza e grandi deviazioni}).

\end{document}